\documentclass[11pt]{article}

\usepackage{amsmath, amssymb, amsfonts}
\usepackage{mathtools}
\usepackage{amsthm}
\usepackage{bm}
\usepackage{bbm}  
\usepackage{yhmath}
\usepackage{geometry}
\usepackage[english]{babel}
\usepackage[babel]{microtype}

\usepackage[shortlabels]{enumitem}

\usepackage{url}
\usepackage{csquotes}

\usepackage{hyperref}
\usepackage[capitalise]{cleveref}

\newtheorem{theorem}{Theorem}[section]
\newtheorem{prop}[theorem]{Proposition}
\newtheorem{lemma}[theorem]{Lemma}
\newtheorem{cor}[theorem]{Corollary}

\newtheorem{claim}[theorem]{Claim}

\theoremstyle{definition}
\newtheorem{defn}[theorem]{Definition}
\newtheorem*{defn-non}{Definition}

\newlist{Case}{enumerate}{2}
\setlist[Case, 1]{%
    label           =   {\bfseries Case \arabic*.},
    labelindent=1em ,labelwidth=1.3cm, labelsep*=1em, leftmargin =!
}
\setlist[Case, 2]{%
    label           =   {\bfseries Subcase \arabic{Casei}.\arabic*.},
    labelindent=-1em ,labelwidth=1.3cm, labelsep*=1em, leftmargin =!
}

\newenvironment{poc}{\begin{proof}[Proof of the claim]}{\end{proof}}

\counterwithin{figure}{section}

\title{Monochromatic unit equilateral triangle on low-dimensional spheres}

\author{
Xiaochen Zhao\thanks{School of Mathematical Sciences, Capital Normal University, Beijing, China. Email: 3535935416@qq.com.}
\and
Gennian Ge\thanks{School of Mathematical Sciences, Capital Normal University, Beijing, China. Email: gnge@zju.edu.cn. Gennian Ge is supported by the National Key Research and Development Program of China under Grant 2025YFC3409900, the National Natural Science Foundation of China under Grant 12231014, and Beijing Scholars Program.}
}

\date{}

\begin{document}
\maketitle

\begin{abstract}
A result of Matou\v{s}ek and R\"odl in 1995 states that for every $\varepsilon>0$ and every triangle $T$ with circumradius $\rho(T)$, there exists a dimension $n=n(\varepsilon,T)$ such that every $2$-coloring of the $n$-dimensional sphere of radius $\rho(T)+\varepsilon$, namely $\mathbb{S}^{n}(\rho(T)+\varepsilon)$, contains a monochromatic congruent copy of $T$. In this paper, we determine the exact threshold dimension for the unit equilateral triangle on the sphere $\mathbb{S}^{n}(1/\sqrt{2})$: there exists a $2$-coloring of $\mathbb{S}^{2}(1/\sqrt{2})$ with no monochromatic unit equilateral triangle, whereas every $2$-coloring of $\mathbb{S}^{3}(1/\sqrt{2})$ contains one. Along the way, we also establish several further Euclidean Ramsey-type results on low-dimensional spheres, including asymmetric and isosceles variants.
\end{abstract}

\section{Introduction}
\label{sec:introduction}

For an integer $n\ge 1$, a point $\boldsymbol{x}\in \mathbb{R}^{n+1}$, and a real number $r>0$, let
\[
\mathbb{S}^n(\boldsymbol{x},r):=\{\boldsymbol{y}\in\mathbb{R}^{n+1}:\|\boldsymbol{y}-\boldsymbol{x}\|=r\}
\]
be the $n$-dimensional sphere of radius $r$ centered at $\boldsymbol{x}$, and write $\mathbb{S}^n(r):=\mathbb{S}^n(\boldsymbol{0},r)$. Throughout the paper, all distances are Euclidean. Points on spheres are denoted by boldface letters such as
$\boldsymbol{x},\boldsymbol{y},\boldsymbol{z}$. The Euclidean norm is denoted by $\|\cdot\|$, and the inner product by $\langle \cdot, \cdot \rangle$ or simply by a dot, as in $\boldsymbol{x}\cdot\boldsymbol{y}$; the cross product is denoted by $\times$, as in $\boldsymbol{x}\times\boldsymbol{y}$. 
\subsection{Background}

Euclidean Ramsey theory originates in the seminal work of Erd\H{o}s, Graham, Montgomery, Rothschild, Spencer, and Straus~\cite{1973JCTA-erdos-euclidean,erdos1975euclidean2,1975JCT-erdos-euclidean3}. Given a finite set $X\subseteq \mathbb{R}^m$, they asked whether there exists an integer $N=N(X,r)$ such that every $r$-coloring of $\mathbb{R}^N$ contains a monochromatic congruent copy of $X$. A finite set with this property is called a \emph{Ramsey set}. In the same paper, they proved that every Ramsey set must be spherical, that is, must lie on some sphere. They also showed that the unit equilateral triangle is not $2$-Ramsey in the plane, and conjectured that every non-equilateral triangle is $2$-Ramsey in $\mathbb{R}^2$.

This conjecture remains widely open. It is known, however, that various special triangles do satisfy the $2$-Ramsey property in the plane; see, for example, the results of Erd\H{o}s, Graham, Montgomery, Rothschild, Spencer, and Straus~\cite{1975JCT-erdos-euclidean3}, Shader~\cite{1976JCT-shader-right-triangles}, and Jel\'inek, Kyn\v{c}l, Stola\v{r}, and Valla~\cite{2009Comb-jelinek-monochromatic}. More recently, Currier, Moore, and Yip~\cite{2024JCTA-currier-Moore-Yip-3-term-arithmetic-progressions} proved that the degenerate configuration consisting of three collinear points with unit spacing is also $2$-Ramsey.

A major breakthrough was achieved by Frankl and R\"odl~\cite{1990JAMS-frankl-partition}, who proved that every simplex is Ramsey; in fact, their argument yields a substantially stronger exponential Ramsey property in high-dimensional Euclidean space. A related direction, introduced by Graham, is to seek monochromatic copies inside a single sphere rather than in the whole ambient space. Following Graham~\cite{1985-graham-old-new}, one calls a finite set $X$ \emph{sphere Ramsey} if for every $r\ge 2$ there exist a dimension $N$ and a radius $p$ such that every $r$-coloring of the sphere $\mathbb{S}^{N-1}(p)$ contains a monochromatic congruent copy of $X$. For a spherical set $X$, let $\rho(X)$ denote its circumradius. Graham asked whether, for a given spherical configuration $X$, one may take the radius to be arbitrarily close to $\rho(X)$; see also Graham~\cite{1983-graham-sphere,1990-graham-topics}. For simplices, this question was answered in a striking way by Matou\v{s}ek and R\"odl~\cite{1995JCTA}.

\begin{theorem}[Matou\v{s}ek--R\"odl~\cite{1995JCTA}]\label{thm:MR-positive}
Let $X$ be a simplex with circumradius $\rho(X)$. Then for every integer $r\ge 2$ and every $\delta>0$, there exists $N=N(X,r,\delta)$ such that every $r$-coloring of $\mathbb{S}^{N-1}(\rho(X)+\delta)$ contains a monochromatic congruent copy of $X$.
\end{theorem}
Thus every simplex is sphere Ramsey on every slightly enlarged sphere. What makes the above theorem especially interesting is that the same paper also proves a complementary negative result showing that, in general, this enlargement cannot be removed. Later, Frankl and R\"odl~\cite{2004ISJM-frankl-strongRamsey} strengthened Theorem~\ref{thm:MR-positive} in a different direction by proving that every simplex is in fact \emph{strong Ramsey}, that is, the same conclusion continues to hold even for colorings with exponentially many colors.

On the one hand, every simplex becomes Ramsey on spheres once the radius is enlarged by an arbitrarily small amount. On the other hand, at the exact circumradius, a monochromatic copy may already fail to exist in every dimension. From this perspective, it is natural to ask what happens in concrete low dimensions for specific simplices and specific radii.

The purpose of this paper is to complement the high-dimensional theory above by establishing several sharp low-dimensional Ramsey results for special triangles like equilateral triangles and isosceles triangles on spheres. Such triangles have been widely studied in Euclidean Ramsey Theory~\cite{2006Graham,2020BLMS}.

\subsection{Our contributions}
For finite point configurations $A$ and $B$, and a sphere $\mathbb{S}$, we write $\mathbb{S}\rightarrow (A;B)$ if every red--blue coloring of $\mathbb{S}$ contains either a red congruent copy of $A$ or a blue congruent copy of $B$; we write $\mathbb{S}\nrightarrow (A;B)$ for the negation. For $a>0$, let $R_a$ denote the equilateral triangle of side length $a$, let $T_a$ denote a pair of points at distance $a$, and let $I_a$ denote the isosceles triangle with side lengths $a,1,1$. We write $R:=R_1$.

Since $\mathbb{S}^n(r)$ appears as an equator of $\mathbb{S}^{n+1}(r)$, the property
\(
\mathbb{S}^n(r)\rightarrow (A;A)
\)
is monotone in $n$. Accordingly, whenever this property holds for some $n$, it is natural to define the threshold dimension
\[
N(A,r):=\min\{n\ge 1:\mathbb{S}^n(r)\rightarrow (A;A)\}.
\]
Our first objective is to determine this threshold for the equilateral triangle $R$ at the radius $r=1/\sqrt{2}$.

The starting point is the following asymmetric Ramsey theorem on the $2$-sphere.

\begin{theorem}\label{thm:main3}
For every $r>0$ and every $a>0$ satisfying $a/r\le \sqrt{3}$,
\[
\mathbb{S}^2(r)\rightarrow (T_a;R_a).
\]
\end{theorem}

The condition $a/r\le \sqrt{3}$ is best possible, since an equilateral triangle of side length $a$ can be inscribed in $\mathbb{S}^2(r)$ if and only if $a/r\le \sqrt{3}$. As an immediate consequence of Theorem~\ref{thm:main3}, we obtain the following special case, which will be used repeatedly later.

\begin{cor}\label{cor:2sphere-ramsey}
\[
   \mathbb{S}^{2}(1/\sqrt{2})\rightarrow (T_1;R).
\]

\end{cor}

We next turn to the symmetric Ramsey problem for $R$ on spheres of radius $1/\sqrt{2}$.

\begin{theorem}\label{thm:main1}\
\begin{enumerate}
    \item[\textup{(1)}] \(\mathbb{S}^{2}(1/\sqrt{2})\nrightarrow (R;R).\)
      
    \item[\textup{(2)}] \(\mathbb{S}^{3}(1/\sqrt{2})\rightarrow (R;R).\)
\end{enumerate}
\end{theorem}

Thus the threshold dimension exists and is exact:
\(
N(R,1/\sqrt{2})=3.
\)
We also obtain an asymmetric statement on $\mathbb{S}^2(1/\sqrt{2})$ involving the isosceles right triangle $I_{\sqrt{2}}$.

\begin{theorem}\label{thm:main5}
\[
   \mathbb{S}^{2}(1/\sqrt{2})\rightarrow (I_{\sqrt{2}};R).
\]

\end{theorem}

More generally, on $\mathbb{S}^3(1/\sqrt{2})$ every red--blue coloring contains a monochromatic isosceles triangle with prescribed side lengths.

\begin{theorem}\label{thm:main2}
For every $0<a<\sqrt{2}$,
\[
\mathbb{S}^{3}(1/\sqrt{2})\rightarrow (I_a;I_a).
\]
\end{theorem}

We will also use, as an auxiliary input, a recent theorem of Cherkashin and Voronov~\cite{2024DCG-cherkashin-chromatic} on the chromatic number of $2$-spheres, which answers an odd problem in~\cite{1976JAUMS-simmons-chromatic}.

\begin{lemma}[Cherkashin--Voronov~\cite{2024DCG-cherkashin-chromatic}]\label{lem:monoline}
Let $T_a$ denote a pair of points at distance $a$, where $a<2r$.
Then
\(
\mathbb{S}^2(r)\rightarrow (T_a;T_a).
\)
\end{lemma}

\section{Proof of Theorem~\ref{thm:main3}}
We argue by contradiction. Suppose that there exists a red-blue coloring of
\(
\mathbb{S}:=\mathbb{S}^{2}(r)\subseteq \mathbb{R}^{3}
\)
such that
\begin{itemize}
    \item[\textnormal{(H1)}] there is no red pair at distance $a$;
    \item[\textnormal{(H2)}] there is no blue equilateral triangle of side length $a$.
\end{itemize}

Write
\(
t:=\frac{a}{r}.
\)
The basic geometric configuration is an \emph{equilateral diamond}: two equilateral triangles of side length $a$ sharing a common base edge. This parallels Simmons' diamond configuration for the sphere's chromatic number~\cite{1976JAUMS-simmons-chromatic}. Since $t\le \sqrt{3}$, an equilateral triangle of side length $a$ can be embedded in $\mathbb{S}$. Therefore, such a diamond can be folded about the common base edge so that all four vertices lie on $\mathbb{S}$. 
Let $\boldsymbol{v}$ and $\boldsymbol{v}'$ be the two tip vertices; after the folding, their distance is denoted by
\[
    b := \|\boldsymbol{v}-\boldsymbol{v}'\|.
\]
For a point $\boldsymbol{x}\in \mathbb{S}$ and $0\le d\le2r$, we write
\begin{equation}\label{eq:def-P_d-main3}
  \mathcal{P}_d(\boldsymbol{x}) := \{\boldsymbol{y}\in \mathbb{S}: \|\boldsymbol{x}-\boldsymbol{y}\| = d\}  
\end{equation}
for the circle of points on $\mathbb{S}$ at distance $d$ from $\boldsymbol{x}$. 
In this section we are mainly interested in the three values $d=a$, $d=b$, and $d=1$; when $d=1$ we write simply $\mathcal{P}(\boldsymbol{x}) := \mathcal{P}_1(\boldsymbol{x})$.
Let $r_b$ denote the radius of $\mathcal{P}_b(\boldsymbol{x})$. 
A direct computation gives
\begin{equation}\label{eq:def-b}
b=2a\sqrt{\frac{3-t^{2}}{4-t^{2}}}, \qquad r_b = \sqrt{b^{2} - \frac{b^{4}}{4r^2}}.
\end{equation}
Moreover, since $t\in(0,\sqrt{3}]$, we obtain $r_b\neq0$ if and only if $t\notin\{\sqrt{2},\sqrt{3}\}$. 

The following propagation lemma is the key point in the non-degenerate case.

\begin{lemma}\label{lem:propagation-main3}
Assume $t<\sqrt{3}$ and $t\neq\sqrt{2}$, let $\boldsymbol{x}\in \mathbb{S}$ be red. Then
\begin{enumerate}
    \item[\textnormal{(i)}] every point of $\mathcal{P}_{a}(\boldsymbol{x})$ is blue;
    \item[\textnormal{(ii)}] every point of $\mathcal{P}_{b}(\boldsymbol{x})$ is red.
\end{enumerate}
\end{lemma}
\begin{proof}
Part \textnormal{(i)} is immediate from \textnormal{(H1)}. For part \textnormal{(ii)}, let $\boldsymbol{y}\in \mathcal{P}_{b}(\boldsymbol{x})$. By the definition of $b$, there exist points
$\boldsymbol{u},\boldsymbol{v}\in \mathcal{P}_{a}(\boldsymbol{x})$
such that
\(
\triangle \boldsymbol{x}\boldsymbol{u}\boldsymbol{v}\) and
\(
\triangle \boldsymbol{y}\boldsymbol{u}\boldsymbol{v}
\)
are both equilateral of side length $a$. By part \textnormal{(i)}, the points $\boldsymbol{u}$ and $\boldsymbol{v}$ are blue. If $\boldsymbol{y}$ were also blue, then $\{\boldsymbol{y},\boldsymbol{u},\boldsymbol{v}\}$ would be a blue equilateral triangle of side length $a$, contradicting \textnormal{(H2)}. Hence $\boldsymbol{y}$ is red.
\end{proof}
We now split the proof into three cases.

\subsection{The case \(a/r=\sqrt{2}\)}
In this case we scale so that
\(
r=\frac{1}{\sqrt{2}},\)
and \(
a=1.
\)
Then for $\boldsymbol{x},\boldsymbol{y}\in \mathbb{S}$,
\begin{equation}\label{eq:unit-orth-main3}
\|\boldsymbol{x}-\boldsymbol{y}\|=1
\iff
\langle \boldsymbol{x},\boldsymbol{y}\rangle =0.
\end{equation}
Then
\begin{equation}\label{eq:def-P-main3}
\mathcal{P}(\boldsymbol{x})
=\{\boldsymbol{y}\in \mathbb{S}:\|\boldsymbol{x}-\boldsymbol{y}\|=1\}
=\{\boldsymbol{y}\in \mathbb{S}:\langle \boldsymbol{x},\boldsymbol{y}\rangle =0\}
\end{equation}
is the great circle $\mathbb{S}\cap \boldsymbol{x}^{\perp}$, where $\boldsymbol{x}^{\perp}$ denotes the plane through the origin orthogonal to $\boldsymbol{x}$. 
We first establish the following two claims, which will be used repeatedly later.

\begin{claim}\label{claim:red-iff-bluecircle-main3}
A point $\boldsymbol{x}\in \mathbb{S}$ is red if and only if every point of $\mathcal{P}(\boldsymbol{x})$ is blue.
\end{claim}

\begin{poc}
If $\boldsymbol{x}$ is red, then \textnormal{(H1)} implies that every point at distance $1$ from $\boldsymbol{x}$ is blue.

Conversely, assume that $\mathcal{P}(\boldsymbol{x})$ is entirely blue. 
Since $\mathcal{P}(\boldsymbol{x})$ is a great circle of Euclidean radius $1/\sqrt{2}$, 
it contains two points at mutual distance $1$; 
if $\boldsymbol{x}$ were also blue, then together with $\boldsymbol{x}$ these three points 
would form a blue equilateral triangle of side length $1$, contradicting \textnormal{(H2)}.
\end{poc}

\begin{defn}
A great circle $\mathcal{P}(\boldsymbol{k})$ is called \emph{transitive} if for every
$\boldsymbol{u}\in \mathcal{P}(\boldsymbol{k})$, the two points of $\mathcal{P}(\boldsymbol{k})$ at distance $1$ from $\boldsymbol{u}$ have the opposite color from $\boldsymbol{u}$.
\end{defn}

\begin{claim}\label{claim:blue-iff-transitive-main3}
A point $\boldsymbol{x}\in \mathbb{S}$ is blue if and only if the great circle $\mathcal{P}(\boldsymbol{x})$ is transitive.
\end{claim}
\begin{poc}
Assume first that $\boldsymbol{x}$ is blue. Let $\boldsymbol{u}\in \mathcal{P}(\boldsymbol{x})$, and let
$\boldsymbol{u}_{+},\boldsymbol{u}_{-}\in \mathcal{P}(\boldsymbol{x})$
be the two points at distance $1$ from $\boldsymbol{u}$.
If $\boldsymbol{u}$ is blue, then $\boldsymbol{u}_{\pm}$ cannot be blue by \textnormal{(H2)}.
If $\boldsymbol{u}$ is red, then $\boldsymbol{u}_{\pm}$ cannot be red by \textnormal{(H1)}.
Hence $\boldsymbol{u}_{\pm}$ always have the opposite color from $\boldsymbol{u}$, so $\mathcal{P}(\boldsymbol{x})$ is transitive.

Conversely, assume that $\mathcal{P}(\boldsymbol{x})$ is transitive. If $\boldsymbol{x}$ were red, then by Claim~\ref{claim:red-iff-bluecircle-main3} the whole circle $\mathcal{P}(\boldsymbol{x})$ would be blue, contradicting transitivity. Hence $\boldsymbol{x}$ is blue.
\end{poc}
We next choose a convenient coordinate configuration. By rotation we may assume that
\(
\boldsymbol{b}:=\left(0,0,\frac{1}{\sqrt{2}}\right)
\)
is red. Then
\[
\mathcal{C}:=\mathcal{P}(\boldsymbol{b})=\{(x,y,z)\in \mathbb{S}:z=0\}
\]
is entirely blue by Claim~\ref{claim:red-iff-bluecircle-main3}. Therefore, for every
$\boldsymbol{k}\in \mathcal{C}$, the great circle $\mathcal{P}(\boldsymbol{k})$ is transitive by Claim~\ref{claim:blue-iff-transitive-main3}. Set
\[
\boldsymbol{c}:=\left(0,-\frac{1}{\sqrt{2}},0\right)\in \mathcal{C},
\qquad
\mathcal{D}:=\mathcal{P}(\boldsymbol{c})=\{(x,y,z)\in \mathbb{S}:y=0\}.
\]
Then $\mathcal{D}$ is transitive. Consider the six points
\begin{align*}
\boldsymbol{a}_{1}&=\left(-\frac{1}{\sqrt{2}},0,0\right),
&
\boldsymbol{a}_{2}&=\left(\frac{1}{\sqrt{2}},0,0\right),\\
\boldsymbol{b}_{1}&=\left(\frac12,0,\frac12\right),
&
\boldsymbol{b}_{2}&=\left(-\frac12,0,-\frac12\right),\\
\boldsymbol{c}_{1}&=\left(-\frac12,0,\frac12\right),
&
\boldsymbol{c}_{2}&=\left(\frac12,0,-\frac12\right).
\end{align*}
They lie on $\mathcal{D}$, and
\(
\|\boldsymbol{b}_{1}-\boldsymbol{c}_{1}\|=1.
\)
Thus $\boldsymbol{b}_{1}$ and $\boldsymbol{c}_{1}$ have opposite colors. Without loss of generality we may assume that $\boldsymbol{c}_{1}$ is blue. Then transitivity along $\mathcal{D}$ forces
\(
\boldsymbol{b}_{1},\boldsymbol{b}_{2}\) to be red, and
\(
\boldsymbol{c}_{2}\) to be blue. Set
\(
\mathcal{D}':=\mathcal{P}(\boldsymbol{b}_{1}).
\)
Since $\boldsymbol{b}_{1}$ is red, Claim~\ref{claim:red-iff-bluecircle-main3} implies that $\mathcal{D}'$ is entirely blue. We refer to
\(
\mathcal{C}\cup \mathcal{D}'
\)
as the \emph{cross circles}.
For $\alpha\in [0,2\pi)$ let
\[
\boldsymbol{k}(\alpha)
=\left(\frac{\cos \alpha}{\sqrt{2}},\frac{\sin \alpha}{\sqrt{2}},0\right).
\] 
Clearly \(\|\boldsymbol{k}(\alpha)\| = 1/\sqrt{2}\), so \(\boldsymbol{k}(\alpha)\in\mathbb{S}\). 
Moreover, $\mathcal{C}=\{\boldsymbol{k}(\alpha):\alpha\in[0,2\pi)\}$.
Define
\[
    \boldsymbol{d}(\alpha)
=
\left(
\frac{\sin\alpha}{\sqrt{2(1+\sin^{2}\alpha)}},
\frac{-\cos\alpha}{\sqrt{2(1+\sin^{2}\alpha)}},
\frac{-\sin\alpha}{\sqrt{2(1+\sin^{2}\alpha)}}
\right),\]
\begin{equation}\label{eq:defn of d(alpha)'}
\boldsymbol{d}'(\alpha)
=
\left(
\frac{\sin^{2}\alpha}{\sqrt{2(1+\sin^{2}\alpha)}},
\frac{-\sin\alpha\cos\alpha}{\sqrt{2(1+\sin^{2}\alpha)}},
\frac{1}{\sqrt{2(1+\sin^{2}\alpha)}}
\right).
\end{equation}
One checks that
\[
\boldsymbol{d}(\alpha)\in \mathcal{P}(\boldsymbol{k}(\alpha))\cap \mathcal{D}',
\qquad
\boldsymbol{d}'(\alpha)\in \mathcal{P}(\boldsymbol{k}(\alpha)),
\qquad
\|\boldsymbol{d}(\alpha)-\boldsymbol{d}'(\alpha)\|=1.
\]
Since $\boldsymbol{d}(\alpha)\in \mathcal{D}'$ is blue and $\mathcal{P}(\boldsymbol{k}(\alpha))$ is transitive,
$\boldsymbol{d}'(\alpha)$ is red. Hence
\[
\Gamma:=\{\boldsymbol{d}'(\alpha):\alpha\in [0,2\pi)\}
\]
is a closed red curve on $\mathbb{S}$.

At this point the proof becomes genuinely geometric. We present it by several separate lemmas.

\begin{lemma}\label{lem:sqrt2-cap}
Define  
\(\mathcal{S}
:=\left\{\boldsymbol{x}\in\mathbb{S}:\left\|\boldsymbol{x}-\frac{\boldsymbol{b}+\boldsymbol{b}_{1}}2\right\|
=\frac12\|\boldsymbol{b}-\boldsymbol{b}_{1}\|\right\}\). Let
\(
\boldsymbol{q}:=(1,0,1+\sqrt2),
\) and \(
c:=1+\frac{\sqrt2}{2},
\)
and let \(U^\circ\) be the connected component of \(\mathbb S\setminus \Gamma\) containing the open cap
\(
\{\boldsymbol{x}\in \mathbb S:\boldsymbol{q}\cdot \boldsymbol{x}>c\},
\)
and let \(U:=\overline{U^\circ}\).
Then
\(
\mathcal S\subseteq U.
\)
In particular, \(\mathcal S\) lies in the closed region bounded by \(\Gamma\).
\end{lemma}
\begin{proof}
We first rewrite \(\mathcal S\) in a more transparent form. By definition,
\[
\mathcal S
=
\left\{
\boldsymbol{x}\in \mathbb S:
\left\|\boldsymbol{x}-\frac{\boldsymbol{b}+\boldsymbol{b}_{1}}{2}\right\|
=
\frac12\|\boldsymbol{b}-\boldsymbol{b}_{1}\|
\right\},
\]
where
\(
\boldsymbol{b}=\left(0,0,\frac1{\sqrt2}\right)\) and
\(
\boldsymbol{b}_{1}=\left(\frac12,0,\frac12\right).
\)
Since $\|\boldsymbol{x}\|^{2}=1/2$ for every $\boldsymbol{x}\in\mathbb S$, expanding the norm yields
$\boldsymbol{x}\in \mathcal S \iff x+(1+\sqrt2)z=1+\frac{\sqrt2}{2}$,
hence $\mathcal S=\{\boldsymbol{x}\in\mathbb S:\boldsymbol{q}\cdot \boldsymbol{x}=c\}$.
For any $\boldsymbol{d}'(\alpha)\in\Gamma$, using the explicit parametrization~\eqref{eq:defn of d(alpha)'}, we compute
\[
\boldsymbol{q}\cdot \boldsymbol{d}'(\alpha)-c
=
\frac{\sqrt2}{2\sqrt{1+\sin^{2}\alpha}}
\Bigl(\sqrt{1+\sin^{2}\alpha}-1\Bigr)
\Bigl(\sqrt{1+\sin^{2}\alpha}-\sqrt2\Bigr).
\]
Since
\(
1\le \sqrt{1+\sin^{2}\alpha}\le \sqrt2,
\)
it follows that
\(
\boldsymbol{q}\cdot \boldsymbol{d}'(\alpha)\le c
\)
for every \(\alpha\), with equality if and only if \(\sin\alpha=0\) or \(|\sin\alpha|=1\), namely at the two distinguished points \(\boldsymbol{b}\) and \(\boldsymbol{b}_{1}\). Hence \(\Gamma\) is contained in the closed cap
\(
\{\boldsymbol{x}\in\mathbb S:\boldsymbol{q}\cdot \boldsymbol{x}\le c\},
\)
and it meets the boundary latitude \(\mathcal S\) only at \(\boldsymbol{b}\) and \(\boldsymbol{b}_{1}\).

Finally, note that $\boldsymbol{d}'(\alpha+\pi)=\boldsymbol{d}'(\alpha)$ for all $\alpha$, so the parametrization has period $\pi$.  
On the interval $[0,\pi)$ the map $\alpha\mapsto\boldsymbol{d}'(\alpha)$ is injective: the $x$- and $z$-coordinates determine $\sin^2\alpha$, and the sign of the $y$-coordinate distinguishes $\alpha$ from $\pi-\alpha$.  
Hence $\Gamma$ is a simple closed curve on $\mathbb{S}$.  
By the Jordan curve theorem for the sphere, $\Gamma$ separates $\mathbb{S}$ into two open regions, each homeomorphic to an open disk.  
Let $U^\circ$ be the component that contains the open cap $\{\boldsymbol{x}\in\mathbb{S}:\boldsymbol{q}\cdot\boldsymbol{x}>c\}$, and set $U:=\overline{U^\circ}$.  
Since $\Gamma$ lies in the closed cap $\{\boldsymbol{x}:\boldsymbol{q}\cdot\boldsymbol{x}\le c\}$ and meets its boundary $\mathcal{S}$ only at $\boldsymbol{b}$ and $\boldsymbol{b}_1$, the whole circle $\mathcal{S}$ is contained in the closure of $U^\circ$; consequently $\mathcal{S}\subseteq U$.  
\end{proof}

\begin{lemma}\label{lem:sqrt2-strip}
Define
\(
\mathbb{E}
:=
\{(x,y,z)\in \mathbb{S}:1/\sqrt{2}-1\le z-(1-\sqrt{2})x\le 1-1/\sqrt{2}\}.
\)
Then \(\mathbb{E}\) is a closed spherical strip, invariant under rotations about the axis
\[
\mathcal{L}:=\{(x,y,z)\in \mathbb{R}^{3}:y=0,\ z=(1+\sqrt{2})x\},
\]
and
\(
\mathbb{E}=\bigcup_{\boldsymbol{x}\in \mathcal{S}}\mathcal{P}(\boldsymbol{x}).
\)
\end{lemma}

\begin{proof}
Let
\(
\boldsymbol{q}:=(1,0,1+\sqrt2),\) and
\(
\boldsymbol{e}:=\frac{\boldsymbol{q}}{\|\boldsymbol{q}\|}.
\)
Then \(\boldsymbol{e}\) is the unit vector along the axis \(\mathcal L\). By the proof of Lemma~\ref{lem:sqrt2-cap},
\(
\mathcal S=\{\boldsymbol{x}\in\mathbb S:\boldsymbol{q}\cdot \boldsymbol{x}=c\}\) and
\(
c=1+\frac{\sqrt2}{2}.
\)
Meanwhile, we also have
\(
\mathcal S=\{\boldsymbol{x}\in\mathbb S:\boldsymbol{e}\cdot \boldsymbol{x}=h\}\) and
\(
h:=\frac{c}{\|\boldsymbol{q}\|}.
\)

We determine \(\bigcup_{\boldsymbol{x}\in\mathcal S}\mathcal P(\boldsymbol{x})\).
Fix \(\boldsymbol{y}\in\mathbb S\), and write
\(
\boldsymbol{x}=h\boldsymbol{e}+\boldsymbol{u},\)
and
\(
\boldsymbol{y}=s\boldsymbol{e}+\boldsymbol{v},
\)
with \(\boldsymbol{u},\boldsymbol{v}\perp \boldsymbol{e}\). Since \(\|\boldsymbol{x}\|^{2}=\|\boldsymbol{y}\|^{2}=1/2\), we have
\(
\|\boldsymbol{u}\|^{2}=\frac12-h^{2}\) and
\(
\|\boldsymbol{v}\|^{2}=\frac12-s^{2}.
\)
Now \(\boldsymbol{y}\in \mathcal P(\boldsymbol{x})\) is equivalent to \(\boldsymbol{x}\cdot \boldsymbol{y}=0\), that is,
\(
hs+\boldsymbol{u}\cdot \boldsymbol{v}=0.
\)
For fixed \(\boldsymbol{y}\), such a vector \(\boldsymbol{u}\) with \(\|\boldsymbol{u}\|^{2}=1/2-h^{2}\) exists if and only if
\[
|hs|\le \|\boldsymbol{u}\|\,\|\boldsymbol{v}\|
=
\sqrt{\frac12-h^{2}}\sqrt{\frac12-s^{2}}.
\]
Squaring and simplifying gives
\(
s^{2}\le \frac12-h^{2}.
\)
Therefore
\[
\bigcup_{\boldsymbol{x}\in\mathcal S}\mathcal P(\boldsymbol{x})
=
\left\{
\boldsymbol{y}\in\mathbb S:
|\boldsymbol{e}\cdot \boldsymbol{y}|
\le
\sqrt{\frac12-h^{2}}
\right\}.
\]

It remains to rewrite this in coordinates. Since
\(
\|\boldsymbol{q}\|^{2}=4+2\sqrt2,
\)
and \(
h=\frac{1+\sqrt2/2}{\sqrt{4+2\sqrt2}},
\)
we obtain
\[
\|\boldsymbol{q}\|\sqrt{\frac12-h^{2}}=\frac1{\sqrt2}.
\]
Hence the above condition becomes
\(
|\boldsymbol{q}\cdot \boldsymbol{y}|\le \frac1{\sqrt2}.
\)
Because
\[
\boldsymbol{q}\cdot (x,y,z)=x+(1+\sqrt2)z
=
(1+\sqrt2)\bigl(z-(1-\sqrt2)x\bigr),
\]
we obtain
\[
1/\sqrt2-1
\le
z-(1-\sqrt2)x
\le
1-1/\sqrt2.
\]
This is exactly the definition of \(\mathbb E\). Finally, \(\mathbb E\) is a spherical strip bounded by two planes orthogonal to \(\mathcal L\), and is therefore invariant under rotations about \(\mathcal L\).
\end{proof}

\begin{claim}\label{lem:sqrt2-orth}
\(
\mathbb{E}\subseteq \bigcup_{\boldsymbol{u}\in \Gamma}\mathcal{P}(\boldsymbol{u}).
\)
In particular, every point of \(\mathbb{E}\) is blue, and every point of \(\mathcal{S}\) is red.
\end{claim}

\begin{poc}
Fix any \(\boldsymbol{y}=(x,y,z)\in \mathbb{E}\). By Lemma~\ref{lem:sqrt2-strip}, the condition \(\boldsymbol{y}\in\mathbb{E}\) is equivalent to
\(|x+(1+\sqrt2)z|\le 1/\sqrt2\).
To show \(\boldsymbol{y}\in\bigcup_{\boldsymbol{u}\in\Gamma}\mathcal{P}(\boldsymbol{u})\), it suffices to find \(\alpha\) such that \(\boldsymbol{d}'(\alpha)\cdot\boldsymbol{y}=0\).
From the definition~\eqref{eq:defn of d(alpha)'} of \(\boldsymbol{d}'(\alpha)\),
\[
\boldsymbol{d}'(\alpha)\cdot\boldsymbol{y}=
\frac{1}{\sqrt{2(1+\sin^2\alpha)}}\bigl(x\sin^2\alpha - y\sin\alpha\cos\alpha + z\bigr).
\]
Using $\sin^2\alpha=\frac{1-\cos2\alpha}{2}$ and $\sin\alpha\cos\alpha=\frac12\sin2\alpha$, the equation $\boldsymbol{d}'(\alpha)\cdot\boldsymbol{y}=0$ becomes
\[
x\cos2\alpha + y\sin2\alpha = x+2z. \tag{*}
\]
Equation~(*) has a solution for \(\alpha\) if and only if
\[
x^2+y^2 \ge (x+2z)^2.
\]
Since \(\boldsymbol{y}\in\mathbb{S}^2(1/\sqrt2)\), we have \(x^2+y^2 = \frac12 - z^2\). Hence the solvability condition is
\[
\frac12 - z^2 \ge (x+2z)^2. \tag{†}
\]

We now prove that (†) follows from \(\boldsymbol{y}\in\mathbb{E}\).
Squaring \(|x+(1+\sqrt2)z|\le 1/\sqrt2\) and substituting \(x^2 = \frac12 - y^2 - z^2\) gives
\[
- y^2 + 2(1+\sqrt2)z(x+z) \le 0,
\qquad\text{i.e.}\qquad
2(1+\sqrt2)z(x+z) \le y^2.
\]
Because \(4 < 2(1+\sqrt2)\), we obtain \(4z(x+z) \le y^2\), i.e. \(y^2 \ge 4xz+4z^2\).
Adding \(x^2\) to both sides and using \(x^2+y^2 = \frac12 - z^2\) yields exactly (†).
Therefore (†) holds, and there exists \(\alpha\) with \(\boldsymbol{d}'(\alpha)\cdot\boldsymbol{y}=0\).
Consequently \(\boldsymbol{y}\in\mathcal{P}(\boldsymbol{d}'(\alpha))\subseteq\bigcup_{\boldsymbol{u}\in\Gamma}\mathcal{P}(\boldsymbol{u})\).

This proves \(\mathbb{E}\subseteq\bigcup_{\boldsymbol{u}\in\Gamma}\mathcal{P}(\boldsymbol{u})\).
Since every point of \(\Gamma\) is red, Claim~\ref{claim:red-iff-bluecircle-main3} implies \(\mathcal{P}(\boldsymbol{u})\) is entirely blue for every \(\boldsymbol{u}\in\Gamma\); hence \(\mathbb{E}\) is blue.
Finally, if \(\boldsymbol{x}\in\mathcal{S}\), then Lemma~\ref{lem:sqrt2-strip} gives \(\mathcal{P}(\boldsymbol{x})\subseteq\mathbb{E}\), so \(\mathcal{P}(\boldsymbol{x})\) is entirely blue, and Claim~\ref{claim:red-iff-bluecircle-main3} forces \(\boldsymbol{x}\) to be red.
\end{poc}

\begin{prop}\label{lem:sqrt2-enlargement}
Starting from the red circle \(\mathcal S\) and repeating the enlargement step twice,
one obtains a red circle \(\mathcal S_2\) around the axis \(\mathcal L\) whose Euclidean
diameter is strictly larger than \(1\).
\end{prop}
\begin{proof}
Let
\(
r=\frac1{\sqrt2},\) and
\(
\boldsymbol{e}:=\frac{1}{\sqrt{4+2\sqrt2}}\,(1,0,1+\sqrt2),
\)
so that $\boldsymbol{e}$ is the unit vector along the axis $\mathcal L$.
Choose also
\[
\boldsymbol{u}:=\frac{1}{\sqrt{4+2\sqrt2}}\,(1+\sqrt2,0,-1),
\qquad
\boldsymbol{v}:=(0,1,0).
\]
Then $\{\boldsymbol{u},\boldsymbol{v},\boldsymbol{e}\}$ is an orthonormal basis of $\mathbb{R}^{3}$, and the great circle
\(
\mathcal D=\{(x,y,z)\in\mathbb S:y=0\}
\)
satisfies exactly
\(
\mathcal D=\mathbb S\cap \operatorname{span}\{\boldsymbol{u},\boldsymbol{e}\}.
\)

For every $h\in[0,r]$, write
\[
a(h):=\sqrt{r^{2}-h^{2}},
\qquad
\mathcal S(h):=\{\boldsymbol{x}\in\mathbb S:\boldsymbol{e}\cdot \boldsymbol{x}=h\}.
\]
Thus $\mathcal S(h)$ is the circle obtained by rotating the two points
\(
a(h)\boldsymbol{u}+h\boldsymbol{e},
\) and
\(
-a(h)\boldsymbol{u}+h\boldsymbol{e}
\)
about the axis $\mathcal L$.

We first record the general form of the blue strip generated by $\mathcal S(h)$.

\begin{claim}
For every $h\in[0,r]$,
\[
\bigcup_{\boldsymbol{x}\in\mathcal S(h)}\mathcal P(\boldsymbol{x})
=
\mathbb E(h)
:=
\{\boldsymbol{y}\in\mathbb S: |\boldsymbol{e}\cdot \boldsymbol{y}|\le a(h)\}.
\]
\end{claim}
\begin{poc}
This is the same computation as in Lemma~\ref{lem:sqrt2-strip}, but with a general height parameter.
Fix $\boldsymbol{y}\in\mathbb S$, and decompose
\(
\boldsymbol{x}=h\boldsymbol{e}+\boldsymbol{\xi},
\)
and
\(
\boldsymbol{y}=s\boldsymbol{e}+\boldsymbol{\eta},
\)
with $\boldsymbol{\xi},\boldsymbol{\eta}\perp \boldsymbol{e}$.
Then
\[
\|\boldsymbol{\xi}\|^{2}=r^{2}-h^{2}=a(h)^{2},
\qquad
\|\boldsymbol{\eta}\|^{2}=r^{2}-s^{2}.
\]
Now $\boldsymbol{y}\in\mathcal P(\boldsymbol{x})$ is equivalent to
\[
\boldsymbol{x}\cdot \boldsymbol{y}=0
\iff
hs+\boldsymbol{\xi}\cdot \boldsymbol{\eta}=0.
\]
For fixed $\boldsymbol{y}$, such a vector $\boldsymbol{\xi}$ with $\|\boldsymbol{\xi}\|=a(h)$ exists if and only if
\[
|hs|\le a(h)\sqrt{r^{2}-s^{2}}.
\]
Squaring and simplifying gives
\(
s^{2}\le a(h)^{2}.
\)
Since $s=\boldsymbol{e}\cdot \boldsymbol{y}$, this proves the claim.
\end{poc}

We now isolate one enlargement step.
\begin{claim}
Assume that $h\in[1/2,r)$ and that the circle $\mathcal S(h)$ is red.
Then the circle
\(
\mathcal S\!\left(\sqrt{2h^{2}-r^{2}}\right)
\)
is also red.
\end{claim}
\begin{poc}
Since $\mathcal S(h)$ is red, Claim~\ref{claim:red-iff-bluecircle-main3} implies that
$\mathcal P(\boldsymbol{x})$ is entirely blue for every $\boldsymbol{x}\in\mathcal S(h)$.
By the previous claim, the whole strip
\[
\mathbb E(h)=\{\boldsymbol{y}\in\mathbb S:|\boldsymbol{e}\cdot \boldsymbol{y}|\le a(h)\}
\]
is therefore blue. Set $a:=a(h)=\sqrt{r^{2}-h^{2}}$ and define
\[
\boldsymbol{x}_{h}:=-a\boldsymbol{u}+h\boldsymbol{e}\in \mathcal S(h)\cap \mathcal D,
\qquad
\boldsymbol{k}_{h}:=h\boldsymbol{u}+a\boldsymbol{e}\in \mathcal D.
\]
Since
\(
\boldsymbol{x}_{h}\cdot \boldsymbol{k}_{h}= -ah+ha =0,
\)
we have $\boldsymbol{k}_{h}\in \mathcal P(\boldsymbol{x}_{h})\subseteq \mathbb E(h)$, so $\boldsymbol{k}_{h}$ is blue.
Hence the great circle $\mathcal P(\boldsymbol{k}_{h})$ is transitive by Claim~\ref{claim:blue-iff-transitive-main3}.

Now define
\[
\boldsymbol{f}_{h}
:=
-\frac{a^{2}}{h}\boldsymbol{u}
-\frac{r\sqrt{2h^{2}-r^{2}}}{h}\boldsymbol{v}
+a\boldsymbol{e}.
\]
Because $h\ge 1/2$ and $r^{2}=1/2$, the quantity $2h^{2}-r^{2}$ is nonnegative, so this is well-defined.
A direct computation gives
\[
\|\boldsymbol{f}_{h}\|^{2}
=
\frac{a^{4}}{h^{2}}+\frac{r^{2}(2h^{2}-r^{2})}{h^{2}}+a^{2}
=
r^{2},
\]
and
\[
\boldsymbol{k}_{h}\cdot \boldsymbol{f}_{h}
=
h\left(-\frac{a^{2}}{h}\right)+a\cdot a
=
0.
\]
Thus $\boldsymbol{f}_{h}\in \mathcal P(\boldsymbol{k}_{h})$.
Moreover,
\(
\boldsymbol{e}\cdot \boldsymbol{f}_{h}=a,
\)
so $\boldsymbol{f}_{h}\in \partial \mathbb E(h)$.
Since $\mathbb E(h)$ is blue, the point $\boldsymbol{f}_{h}$ is blue. Define
\[
\boldsymbol{f}_{h}'
:=
-\frac1r\,\boldsymbol{k}_{h}\times \boldsymbol{f}_{h}.
\]
Because $\boldsymbol{k}_{h}$ and $\boldsymbol{f}_{h}$ are orthogonal and both have norm $r$, the vector
$\boldsymbol{f}_{h}'$ also has norm $r$, and it is orthogonal to both
$\boldsymbol{k}_{h}$ and $\boldsymbol{f}_{h}$.
Hence
\(
\boldsymbol{f}_{h}'\in \mathcal P(\boldsymbol{k}_{h})\)
and
\(
\boldsymbol{f}_{h}\cdot \boldsymbol{f}_{h}'=0.
\)
Since $r=1/\sqrt2$, the identity $\boldsymbol{f}_{h}\cdot \boldsymbol{f}_{h}'=0$ is equivalent to
\(
\|\boldsymbol{f}_{h}-\boldsymbol{f}_{h}'\|=1.
\)
Therefore, by transitivity of $\mathcal P(\boldsymbol{k}_{h})$, the point
$\boldsymbol{f}_{h}'$ must be red.

Writing everything in the orthonormal basis
$\{\boldsymbol{u},\boldsymbol{v},\boldsymbol{e}\}$, one computes
\[
\boldsymbol{f}_{h}'
=
-\frac{a\sqrt{2h^{2}-r^{2}}}{h}\boldsymbol{u}
+\frac{ra}{h}\boldsymbol{v}
+\sqrt{2h^{2}-r^{2}}\,\boldsymbol{e}.
\]
In particular,
\(
\boldsymbol{e}\cdot \boldsymbol{f}_{h}'=\sqrt{2h^{2}-r^{2}}.
\)
For every rotation \(\rho\) about the axis \(\mathcal L\), the points
\[
\rho(\boldsymbol{x}_h)\in \mathcal S(h),\qquad
\rho(\boldsymbol{k}_h)\in \mathbb E(h),\qquad
\rho(\boldsymbol{f}_h)\in \partial \mathbb E(h)
\]
satisfy exactly the same geometric relations as
\(\boldsymbol{x}_h,\boldsymbol{k}_h,\boldsymbol{f}_h\).
Since \(\mathcal S(h)\) is entirely red and \(\mathbb E(h)\) is entirely blue, the same argument as above shows that
\(\rho(\boldsymbol{f}_h')\) is red.
As \(\rho\) ranges over all rotations about \(\mathcal L\), the points \(\rho(\boldsymbol{f}_h')\) trace precisely the circle
\(
\mathcal S\!\left(\sqrt{2h^{2}-r^{2}}\right).
\)
Hence every point of this circle is red.
This proves the enlargement step.
\end{poc}

We now apply the enlargement step twice.

By Lemma~\ref{lem:sqrt2-cap}, the original red circle $\mathcal S$ is of the form
\[
\mathcal S=\mathcal S(h_{0}),
\qquad
h_{0}=\frac{c}{\|\boldsymbol{q}\|}
=\frac{1+\sqrt2}{2\sqrt{2+\sqrt2}}.
\]
Hence
\(
h_{0}^{2}=\frac{2+\sqrt2}{8}>\frac14.
\)
Applying the enlargement step once gives a red circle
\[
\mathcal S_{1}=\mathcal S(h_{1}),
\qquad
h_{1}^{2}=2h_{0}^{2}-r^{2}=\frac{\sqrt2}{4}.
\]
Since \(h_{1}^{2}>\frac14\), we may apply the enlargement step once more and obtain a red circle
\[
\mathcal S_{2}=\mathcal S(h_{2}),
\qquad
h_{2}^{2}=2h_{1}^{2}-r^{2}=\frac{\sqrt2-1}{2}.
\]

Finally, the Euclidean radius of the circle $\mathcal S(h)$ equals $\sqrt{r^{2}-h^{2}}$, so
\[
\operatorname{diam}(\mathcal S_{2})
=
2\sqrt{r^{2}-h_{2}^{2}}
=
2\sqrt{\frac12-\frac{\sqrt2-1}{2}}
=
\sqrt{4-2\sqrt2}>1.
\]
This completes the proof.
\end{proof}

\begin{proof}[Completion of the case \(a/r=\sqrt{2}\)]
By Claim~\ref{lem:sqrt2-orth}, the whole strip $\mathbb{E}$ is blue, while $\mathcal{S}$ is a red circle.
Proposition~\ref{lem:sqrt2-enlargement} then yields a red circle $\mathcal{S}_{2}$ of Euclidean diameter larger than $1$.
Hence $\mathcal{S}_{2}$ contains two red points at mutual distance exactly $1$, contradicting \textnormal{(H1)}.
\end{proof}

\subsection{The case \(a/r=\sqrt{3}\)}\label{subsec:2.2}

In this case we scale so that
\(
r=\frac{1}{\sqrt{3}},\)
and \(
a=1.
\)
Then for $\boldsymbol{x},\boldsymbol{y}\in \mathbb{S}$,
\[
\|\boldsymbol{x}-\boldsymbol{y}\|=1
\iff
\langle \boldsymbol{x},\boldsymbol{y}\rangle =-\frac16.
\]
Then
\[
\mathcal{P}(\boldsymbol{x})
=
\{\boldsymbol{y}\in \mathbb{S}:\|\boldsymbol{x}-\boldsymbol{y}\|=1\}
=
\{\boldsymbol{y}\in \mathbb{S}:\langle \boldsymbol{x},\boldsymbol{y}\rangle =-1/6\}
\]
is a circle with center \(-\frac{\boldsymbol{x}}{2}\) and Euclidean radius \(\frac12\), any two antipodal points on $\mathcal{P}(\boldsymbol{x})$ are at mutual distance $1$.

Exactly as above, we can obtain the following two claims and we omit the repeated details.

\begin{claim}\label{claim:sqrt3-red}
A point $\boldsymbol{x}\in \mathbb{S}$ is red if and only if every point of $\mathcal{P}(\boldsymbol{x})$ is blue.
\end{claim}

\begin{claim}\label{claim:sqrt3-blue}
A point $\boldsymbol{x}\in \mathbb{S}$ is blue if and only if every antipodal pair on $\mathcal{P}(\boldsymbol{x})$ receives opposite colors.
\end{claim}

Fix
\(
\boldsymbol{b}:=\left(0,0,\frac1{\sqrt3}\right)
\)
red. Then $\mathcal{C}:=\mathcal{P}(\boldsymbol{b})$ is entirely blue. Choose
\[
\boldsymbol{c}:=\left(0,-\frac12,-\frac1{2\sqrt3}\right)\in \mathcal{C}.
\]
Then $\mathcal{D}:=\mathcal{P}(\boldsymbol{c})$ is transitive in the above antipodal sense. Consider
\[
\boldsymbol{a}_{1}
=
\left(-\frac12,\frac14,\frac1{4\sqrt3}\right),
\qquad
\boldsymbol{a}_{2}
=
\left(\frac12,\frac14,\frac1{4\sqrt3}\right).
\]
These are antipodal on $\mathcal{D}$, so they have opposite colors. Without loss of generality, assume that $\boldsymbol{a}_{1}$ is blue and $\boldsymbol{a}_{2}$ is red.
Set
\(
\mathcal{D}':=\mathcal{P}(\boldsymbol{a}_{2}).
\)
Then $\mathcal{D}'$ is entirely blue by Claim~\ref{claim:sqrt3-red}. 

Let \(\alpha\in[0,2\pi)\) and define
\[
\boldsymbol{k}(\alpha) = \left( \frac{\cos\alpha}{2},\; \frac{\sin\alpha}{2},\; -\frac{1}{2\sqrt{3}} \right).
\]
Clearly \(\|\boldsymbol{k}(\alpha)\| = 1/\sqrt{3}\), so \(\boldsymbol{k}(\alpha)\in\mathbb{S}\). Moreover, $\mathcal{C}=\{\boldsymbol{k}(\alpha):\alpha\in[0,2\pi)\}$. Note that 
\[
\mathcal{P}(\boldsymbol{k}(\alpha)) = \left\{ \boldsymbol{x}\in\mathbb{S} : \boldsymbol{x}\cdot\boldsymbol{k} = -\tfrac16 \right\}
= \left\{ (x,y,z)\in\mathbb{S}:
3\cos\alpha\,x + 3\sin\alpha\,y - \sqrt{3}\,z = -1 \right\},
\]
and
\[
\mathcal{D}' := \mathcal{P}(\boldsymbol{a}_2) = \left\{ \boldsymbol{x}\in\mathbb{S} : \boldsymbol{x}\cdot\boldsymbol{a}_2 = -\tfrac16 \right\}
= \left\{ (x,y,z)\in\mathbb{S}:
6x+3y+\sqrt{3}\,z = -2 \right\}.
\]
The circles \(\mathcal{P}(\boldsymbol{k}(\alpha))\) and \(\mathcal{D}'\) lie in the planes
\[
\Pi_1:\; \boldsymbol{x}\cdot\boldsymbol{k}(\alpha)=-\frac16,\qquad
\Pi_2:\; \boldsymbol{x}\cdot\boldsymbol{a}_2=-\frac16.
\]
Let \(\theta\) be the angle between the normals \(\boldsymbol{k}(\alpha)\) and \(\boldsymbol{a}_2\). 
The distance from the center of the sphere $\mathbb{S}$ to the line of intersection of the two planes is \(d = \frac{\sqrt{3}}{6\cos(\theta/2)}\). The two circles intersect precisely when this line meets the sphere, that is, when \(d \le r = 1/\sqrt3\). This condition is equivalent to \(1+2\cos\alpha+\sin\alpha \ge 0\), which gives
\[
\alpha \in [0,\,2\arctan 3]\;\cup\;[\tfrac{3\pi}{2},\,2\pi).
\]
Hence the two circles \(\mathcal{P}(\boldsymbol{k})\) and \(\mathcal{D}'\) intersect if and only if $\alpha\in[0,2\arctan{3}]\cup[3\pi/2,2\pi)$. In the following, we assume that $\alpha\in[0,2\arctan{3}]\cup[3\pi/2,2\pi)$. 
The intersection points $\boldsymbol{x}$ of $\mathcal{P}(\boldsymbol{k})$ and $\mathcal{D}'$ satisfy the system
\begin{equation}\label{eq:intersection of sqrt3}
    \boldsymbol{x}\cdot\boldsymbol{k}(\alpha) = -\frac{1}{6},\qquad \boldsymbol{x}\cdot\boldsymbol{a}_2=-\frac{1}{6}, \qquad \|\boldsymbol{x}\|^2 = r^2=\frac{1}{3}.
\end{equation}
Solving the system~\eqref{eq:intersection of sqrt3} (the elementary algebra is summarized in~\ref{app:sqrt3-derivation}) yields the two intersection points $\{\boldsymbol{d}_1,\boldsymbol{d}_2\}=\mathcal{P}(\boldsymbol{k})\cap\mathcal{D}'$, where
\[
\begin{aligned}
\boldsymbol{d}_1 &= \left(
\frac{-2(1+\cos\alpha) - \dfrac{\Delta}{\Sigma}\,(1+\sin\alpha)}{D},\;
\frac{-(2\sin\alpha+1) + \dfrac{\Delta}{\Sigma}\,(2+\cos\alpha)}{D},\;
\frac{1 - \dfrac{3\Delta}{\Sigma}\,(\cos\alpha-2\sin\alpha)}{\sqrt{3}\,D}
\right),\\[10pt]
\boldsymbol{d}_2 &= \left(
\frac{-2(1+\cos\alpha) + \dfrac{\Delta}{\Sigma}\,(1+\sin\alpha)}{D},\;
\frac{-(2\sin\alpha+1) - \dfrac{\Delta}{\Sigma}\,(2+\cos\alpha)}{D},\;
\frac{1 + \dfrac{3\Delta}{\Sigma}\,(\cos\alpha-2\sin\alpha)}{\sqrt{3}\,D}
\right),
\end{aligned} 
\]
and
\[
D = 7 + 6\cos\alpha + 3\sin\alpha,\qquad
\Delta = \sqrt{1+2\cos\alpha+\sin\alpha},\qquad
\Sigma = \sqrt{3-2\cos\alpha-\sin\alpha}.
\]
Then both \(\boldsymbol{d}_{1},\boldsymbol{d}_{2}\) are blue.

Define
\(
\boldsymbol{d}_1'\) to be the antipodal point of \(\boldsymbol{d}_1\) on \(\mathcal{P}(\boldsymbol{k}),\)
and similarly \(\boldsymbol{d}_2'\). Then \(\boldsymbol{d}_1'(\alpha)\) and \(\boldsymbol{d}_2'(\alpha)\) must be red by Claim~\ref{claim:sqrt3-blue}. For any \(\boldsymbol{x}\in\mathcal{P}(\boldsymbol{k})\), its antipodal point on the small circle \(\mathcal{P}(\boldsymbol{k})\) is \(-\boldsymbol{k} - \boldsymbol{x}\). Hence
\(
\boldsymbol{d}_1' = -\boldsymbol{k} - \boldsymbol{d}_1\) and
\(
\boldsymbol{d}_2' = -\boldsymbol{k} - \boldsymbol{d}_2.
\)
More explicitly,
\[
\begin{aligned}
\boldsymbol{d}_1' &= \left(
-\frac{\cos\alpha}{2} - x_1,\; -\frac{\sin\alpha}{2} - y_1,\; \frac{1}{2\sqrt{3}} - z_1
\right),\quad
\boldsymbol{d}_2' &= \left(
-\frac{\cos\alpha}{2} - x_2,\; -\frac{\sin\alpha}{2} - y_2,\; \frac{1}{2\sqrt{3}} - z_2
\right),
\end{aligned}
\]
where \((x_i,y_i,z_i)\) are the components of \(\boldsymbol{d}_i\) given above. 
One checks that $\boldsymbol{d}_1'(2\arctan 3)=\boldsymbol{d}_1'(3\pi/2)$ (see~\ref{app:sqrt3-derivation}); hence the map $\alpha\mapsto\boldsymbol{d}_1'(\alpha)$ on $[0,2\arctan 3]\cup[3\pi/2,2\pi)$ extends to a continuous closed curve
\[
\Gamma := \left\{ \boldsymbol{d}_1'(\alpha) : \alpha\in[0,2\arctan 3]\cup[3\pi/2,2\pi) \right\}
\]
on $\mathbb{S}$.
Let $\boldsymbol{q}:=\boldsymbol{d}_2'(0)$, then
\[
   \boldsymbol{q} = \left(\frac{-5-2\sqrt{3}}{26},\; \frac{1+3\sqrt{3}}{13},\; \frac{11-6\sqrt{3}}{26\sqrt{3}}\right).
\]

\begin{claim}\label{claim:sqrt3-final}
    \(
    \mathcal{P}(\boldsymbol{q})\cap \Gamma\neq\varnothing.
    \)
\end{claim}
\begin{poc}
   Let $\boldsymbol{p}_1:=\boldsymbol{d}_1'(0)\in\Gamma\text{ and  }\boldsymbol{p}_2:=\boldsymbol{d}_1'(\pi/2)\in\Gamma$, then 
   \[
   \boldsymbol{p}_1=\left(\frac{-5+2\sqrt{3}}{26},\; \frac{1-3\sqrt{3}}{13},\; \frac{11+6\sqrt{3}}{26\sqrt{3}}\right),\quad
   \boldsymbol{p}_2=\left(\frac25,-\frac25,-\frac{\sqrt3}{15}\right).
   \]
   Define $f(\boldsymbol{x})=\|\boldsymbol{x}-\boldsymbol{q}\|-1$ on $\Gamma$. Then $f$ is continuous on $\Gamma$. 
   Since $f(\boldsymbol{p}_1)=\sqrt{\frac{12}{13}}-1<0$ and $f(\boldsymbol{p}_2)=\sqrt{ \frac{61+14\sqrt{3}}{65}}-1>0$, the Intermediate Value Theorem yields a point $\boldsymbol{p}^*\in\Gamma$ with $f(\boldsymbol{p}^*)=0$, i.e.\ $\|\boldsymbol{p}^*-\boldsymbol{q}\|=1$.
   By definition $\mathcal{P}(\boldsymbol{q})=\{\boldsymbol{x}\in\mathbb{S}:\|\boldsymbol{x}-\boldsymbol{q}\|=1\}$; 
   hence $\boldsymbol{p}^*\in\mathcal{P}(\boldsymbol{q})$, and therefore $\mathcal{P}(\boldsymbol{q})\cap\Gamma\neq\varnothing$.
\end{poc}

\begin{proof}[Completion of the case \(a/r=\sqrt{3}\)]
By Claim~\ref{claim:sqrt3-final}, there exists
\(
\boldsymbol{p}^*\in \mathcal{P}(\boldsymbol{q})\cap \Gamma.
\)
Since $\boldsymbol{q}=\boldsymbol{d}_2'(0)$ is red, Claim~\ref{claim:sqrt3-red} implies that the whole circle
$\mathcal{P}(\boldsymbol{q})$
is blue. On the other hand, every point of $\Gamma$ must be red. Therefore, $\boldsymbol{p}^*$ can not be colored properly under our initial assumption \textnormal{(H1)} and \textnormal{(H2)}. This contradiction proves the case $a/r=\sqrt{3}$.
\end{proof}

\subsection{The non-degenerate case}\label{subsec:2.3}

Assume now that
\(
0<\frac{a}{r}<\sqrt{3}\) and
\(
\frac{a}{r}\notin \{\sqrt{2},\sqrt{3}\}.
\)
Let
\(
\theta\in (0,\pi)
\)
denote the central angle subtended by a chord of length $a$ on $\mathbb{S}$, so that
\(
a=2r\sin \frac{\theta}{2}.
\)
Since $a/r\le \sqrt{3}$, we have $\theta\le 2\pi/3$.

We first note that there exists a red point on $\mathbb{S}$; indeed, if every point were blue, then any inscribed equilateral triangle of side length $a$ would be blue, contradicting \textnormal{(H2)}.
Fix such a red point and, after rotation, assume that it is
\(
\boldsymbol{d}:=(0,0,r).
\)
By Lemma~\ref{lem:propagation-main3}, the set
\(
\mathcal{R}_{1}:=\mathcal{P}_{b}(\boldsymbol{d})
\)
is a red circle.

Let $\mathcal{D}$ be the great circle obtained by intersecting $\mathbb{S}$ with the plane $y=0$. Then 
\[
\mathcal{D}=\{(x,y,z)\in\mathbb{R}^3:x^2+z^2=r^2,y=0\}.
\]
Let $\gamma\in (0,\pi)$ be the central angle on $\mathcal{D}$ cut out by the two intersection points
\(
\mathcal{P}_{b}(\boldsymbol{x})\cap \mathcal{D},
\)
where $\boldsymbol{x}\in \mathcal{R}_{1}\cap \mathcal{D}$.
Then 
\begin{equation}\label{eq:def-gamma}
    \gamma = 2\arctan \dfrac{r_b}{\sqrt{r^2 - r_b^2}},
\end{equation}
where $r_b$ is the radius of circle $\mathcal{P}_b(\boldsymbol{x})$.
Actually, for every $\boldsymbol{x}\in \mathcal{D}$, the set
\(
\mathcal{P}_{b}(\boldsymbol{x})\cap \mathcal{D}
\)
consists of the two points on $\mathcal{D}$ with central angle $\gamma$. If $\gamma\ge \theta$, then $\mathcal{R}_{1}$ already contains two red points at angular separation $\theta$, hence at Euclidean distance $a$, contradicting \textnormal{(H1)}.
Therefore we may assume that $\gamma<\theta$. Then $\gamma<\pi$. Note that if $b=\sqrt{2}r$, then $r_b=r$ and formula~\eqref{eq:def-gamma} 
would give $\gamma=\pi$, contradicting $\gamma<\pi$. Hence $b\neq\sqrt{2}r$.

Let $\mathcal{C}_{2}$ denote the set of red points obtained on $\mathcal{D}$ after one further propagation step:
\[
\mathcal{C}_{2}
:=
\Bigl(\bigcup_{\boldsymbol{x}\in \mathcal{R}_{1}} \mathcal{P}_{b}(\boldsymbol{x})\Bigr)\cap \mathcal{D}.
\]
\begin{claim}\label{claim:C2-structure}
$\mathcal{C}_2$ is a closed arc on the great circle $\mathcal{D}$ whose angular length is exactly $2\gamma$.
\end{claim}

\begin{poc}
We give a unified computation that covers both $b>\sqrt{2}r$ and $b<\sqrt{2}r$.
Parametrize the circle $\mathcal{R}_1$ by
\[
\boldsymbol{x}(t) = \bigl(r_b\cos t,\; r_b\sin t,\; -\sqrt{r^2 - r_b^2}\bigr),\qquad t\in[0,2\pi),
\]
and the great circle $\mathcal{D}$ ($y=0$) by
\[
\boldsymbol{p}(\varphi) = (r\cos\varphi,\,0,\,r\sin\varphi),\qquad \varphi\in[0,2\pi).
\]
The condition $\boldsymbol{p}(\varphi)\in\mathcal{P}_b(\boldsymbol{x}(t))$ is $\|\boldsymbol{p}(\varphi)-\boldsymbol{x}(t)\|^2=b^2$, which expands to
\[
\boldsymbol{p}(\varphi)\cdot\boldsymbol{x}(t)=r^2-\frac{b^2}{2}.
\]
Substituting the coordinates gives
\[
r_b\cos t\cos\varphi - h\,\sin\varphi = r-\frac{b^2}{2r}, \tag{**}
\]
where $h:=\sqrt{r^2-r_b^2}>0$.  Notice that $r^2-r_b^2 = (r-\frac{b^2}{2r})^2$, hence $h=\bigl|r-\frac{b^2}{2r}\bigr|$.

\noindent\textbf{Case 1:} $b>\sqrt{2}r$.  Then $r-\frac{b^2}{2r}<0$, so $h=-(r-\frac{b^2}{2r})$.  
Equation~(**) becomes
\[
r_b\cos t\cos\varphi - h\sin\varphi = -h.
\]
One finds that $\varphi=\pi/2$ is a solution for every $t$; this is the fixed point $\boldsymbol{d}=(0,0,r)$.  
For $\varphi\neq\pi/2$ one obtains
\[
\varphi(t)=\frac{\pi}{2}-2\arctan\!\Bigl(\frac{r_b\cos t}{h}\Bigr).
\]
Set $k:=r_b/h$.  Then $\varphi_2(t)=\frac{\pi}{2}-2\arctan(k\cos t)$ runs over the interval
$[\frac{\pi}{2}-2\arctan k,\;\frac{\pi}{2}+2\arctan k]$ as $t$ varies.  By~\eqref{eq:def-gamma} $\gamma=2\arctan k$, so this interval is $[\frac\pi2-\gamma,\frac\pi2+\gamma]$.  
Hence $\mathcal{C}_2$ is the closed arc $\{\boldsymbol{p}(\varphi):\varphi\in[\frac\pi2-\gamma,\frac\pi2+\gamma]\}$.

\noindent\textbf{Case 2:} $b<\sqrt{2}r$.  Now $r-\frac{b^2}{2r}>0$ and $h=r-\frac{b^2}{2r}$.  
Equation~(**) reads
\[
r_b\cos t\cos\varphi - h\sin\varphi = h.
\]
We reduce this case to Case~1 by symmetry.  Replace $t$ by $t+\pi$ (so $\cos(t+\pi)=-\cos t$) and $\varphi$ by $-\varphi$:
\[
r_b\cos(t+\pi)\cos(-\varphi)-h\sin(-\varphi) = -r_b\cos t\cos\varphi + h\sin\varphi = -(r_b\cos t\cos\varphi - h\sin\varphi) = -h.
\]
Thus $\psi:=-\varphi$ satisfies exactly the equation of Case~1 with parameter $s:=t+\pi$:
\[
r_b\cos s\cos\psi - h\sin\psi = -h.
\]
Consequently, the solutions in Case~2 are obtained from those in Case~1 by
\[
\varphi(t)=-\psi(t+\pi).
\]
From Case~1 we know that the set of angles $\psi(s)$ as $s$ runs over $[0,2\pi)$ is the arc
$[\frac\pi2-\gamma,\frac\pi2+\gamma]$.  Therefore the corresponding $\varphi(t)$ form the arc
$[-\frac\pi2-\gamma,-\frac\pi2+\gamma]$, which modulo $2\pi$ is the same as
$[\frac{3\pi}{2}-\gamma,\frac{3\pi}{2}+\gamma]$.  Its length is again $2\gamma$.

Thus, in both cases $\mathcal{C}_2$ is a closed arc of angular length exactly $2\gamma$ on the great circle $\mathcal{D}$.
\end{poc}

If $2\gamma\ge \theta$, then $\mathcal{C}_{2}$ already contains two red points at angular separation $\theta$, hence at Euclidean distance $a$, contradicting \textnormal{(H1)}.
Therefore we may assume that $2\gamma<\theta$. For $m\ge 2$, define inductively
\[
\mathcal{C}_{m+1}
:=
\Bigl(\bigcup_{\boldsymbol{x}\in \mathcal{C}_{m}} \mathcal{P}_{b}(\boldsymbol{x})\Bigr)\cap \mathcal{D}.
\]

\begin{claim}\label{claim:arc-growth-main3}
For every $m\ge 2$, the set $\mathcal{C}_{m}$ is a red arc on $\mathcal{D}$ of angular length $m\gamma$.
\end{claim}

\begin{poc}
We already established the case $m=2$. Assume that $\mathcal{C}_{m}$ is a red arc of angular length $m\gamma$.
Parameterize $\mathcal{D}$ by angular coordinate $s$ modulo $2\pi$, so that $\mathcal{C}_{m}$ corresponds to a closed interval
\(
I=[u,v]
\)
of length $m\gamma$.

If $b<\sqrt{2}r$, then $\boldsymbol{x}$ lies on the shorter arc determined by the two intersection points, and each intersection point is at angular distance $\gamma/2$ from $\boldsymbol{x}$. Then for each point $s\in I$, the set
\(
\mathcal{P}_{b}(\boldsymbol{x}(s))\cap \mathcal{D}
\)
consists of the two points with angular coordinates $s-\gamma/2$ and $s+\gamma/2$.
Hence
\[
\mathcal{C}_{m+1}
=
[u-\gamma/2,\ v+\gamma/2],
\]
which is again an arc, now of angular length $(m+1)\gamma$.

If $b>\sqrt{2}r$, then $\boldsymbol{x}$ lies on the longer arc, and each intersection point is at angular distance $\pi-\gamma/2$ from $\boldsymbol{x}$. Then for each point $s\in I$, the set
\(
\mathcal{P}_{b}(\boldsymbol{x}(s))\cap \mathcal{D}
\)
consists of the two points with angular coordinates $s-(\pi-\gamma/2)=s-\pi+\gamma/2=s+\pi+\gamma/2$ and $s+(\pi-\gamma/2)=s+\pi-\gamma/2$.
Hence
\[
\mathcal{C}_{m+1}
=
[u+\pi-\gamma/2,\ v+\pi+\gamma/2],
\]
which is again an arc, now of angular length $(m+1)\gamma$.

Since every point produced in this way lies on some $\mathcal{P}_{b}(\boldsymbol{x})$ with $\boldsymbol{x}$ red, Lemma~\ref{lem:propagation-main3} shows that $\mathcal{C}_{m+1}$ is red.
\end{poc}

Choose
\(
m:=\left\lceil\frac{\theta}{\gamma}\right\rceil.
\)
Then $m\gamma\ge \theta$, so by Claim~\ref{claim:arc-growth-main3}, the red arc $\mathcal{C}_{m}$ contains two points at angular separation exactly $\theta$.
These two points are at Euclidean distance $a$, contradicting \textnormal{(H1)}. This contradiction completes the proof of Theorem~\ref{thm:main3}.

\qed

\section{Proof of Theorem~\ref{thm:main1} and Theorem~\ref{thm:main5}}

\subsection{Proof of Theorem~\ref{thm:main1}(2)}

We prove that
\(
\mathbb{S}^{3}\!\left(1/\sqrt{2}\right)\rightarrow (R;R).
\)
Let
\(
\mathbb{S}^{3}:=\mathbb{S}^{3}\!\left(1/\sqrt{2}\right)\subseteq \mathbb{R}^{4},
\)
and let
\[
\mathbb{S}_{\mathrm{eq}}
:=
\{(x,y,z,w)\in \mathbb{S}^{3}:w=0\}
=
\{(x,y,z,0):x^{2}+y^{2}+z^{2}=1/2\}
\]
be the equatorial \(2\)-sphere. Then \(\mathbb{S}_{\mathrm{eq}}\) is isometric to
\(\mathbb{S}^{2}(1/\sqrt2)\).

Fix an arbitrary red-blue coloring of \(\mathbb{S}^{3}\). We distinguish two cases according to the coloring of antipodal pairs.

First suppose that there exist antipodal points of different colors. After a rotation, we may assume that
\(
\boldsymbol{b}=(0,0,0,1/\sqrt2)\) is colored red and
\(
\boldsymbol{a}=(0,0,0,-1/\sqrt2)\) is colored blue.
Apply Lemma~\ref{lem:monoline} to the induced coloring on \(\mathbb{S}_{\mathrm{eq}}\). Then there exists a monochromatic pair
\(
\{\boldsymbol{p}_{1},\boldsymbol{p}_{2}\}\subseteq \mathbb{S}_{\mathrm{eq}}
\)
with
\(
\|\boldsymbol{p}_{1}-\boldsymbol{p}_{2}\|=1.
\)
Since every point of \(\mathbb{S}_{\mathrm{eq}}\) has fourth coordinate \(0\), we have
\[
\|\boldsymbol{b}-\boldsymbol{p}_{i}\|^{2}
=
\|\boldsymbol{b}\|^{2}+\|\boldsymbol{p}_{i}\|^{2}
=
\frac12+\frac12
=
1,
\]
and similarly \(\|\boldsymbol{a}-\boldsymbol{p}_{i}\|=1\) for \(i=1,2\).
Hence:
if \(\boldsymbol{p}_{1},\boldsymbol{p}_{2}\) are red, then
  \(\{\boldsymbol{b},\boldsymbol{p}_{1},\boldsymbol{p}_{2}\}\) is a red unit equilateral triangle;
 if \(\boldsymbol{p}_{1},\boldsymbol{p}_{2}\) are blue, then
  \(\{\boldsymbol{a},\boldsymbol{p}_{1},\boldsymbol{p}_{2}\}\) is a blue unit equilateral triangle. Thus the first case is settled.

Now suppose that every antipodal pair on \(\mathbb{S}^{3}\) has the same color. If every point is blue, then any inscribed unit equilateral triangle is blue and we are done. Otherwise there exists a red point, and hence its antipode is also red. After a rotation, we may assume that
\(
\boldsymbol{b}=(0,0,0,1/\sqrt2)\) and
\(
\boldsymbol{a}=(0,0,0,-1/\sqrt2)
\)
are both colored red.

Apply Corollary~\ref{cor:2sphere-ramsey} to the induced coloring on \(\mathbb{S}_{\mathrm{eq}}\). Then either:
\begin{itemize}
    \item there exists a red unit pair
    \(\{\boldsymbol{p}_{1},\boldsymbol{p}_{2}\}\subseteq \mathbb{S}_{\mathrm{eq}}\), or
    \item there exists a blue unit equilateral triangle
    \(\{\boldsymbol{q}_{1},\boldsymbol{q}_{2},\boldsymbol{q}_{3}\}\subseteq \mathbb{S}_{\mathrm{eq}}\).
\end{itemize}
In the second case we are done immediately. In the first case,
\(\{\boldsymbol{a},\boldsymbol{p}_{1},\boldsymbol{p}_{2}\}\) is a red unit equilateral triangle, since \(\|\boldsymbol{a}-\boldsymbol{p}_{i}\|=1\) for \(i=1,2\). This finishes the proof.

\qed

\subsection{Proof of Theorem~\ref{thm:main1}(1)}\label{sec:counterexample-S2}

In this subsection we construct an explicit \(2\)-coloring of
\(
\mathbb{S}:=\mathbb{S}^{2}\!\left(1/\sqrt{2}\right)\subseteq \mathbb{R}^{3}
\)
with no monochromatic unit equilateral triangle. 

For \(\boldsymbol{u},\boldsymbol{v}\in \mathbb{S}\), we have
\(
\|\boldsymbol{u}\|^{2}=\|\boldsymbol{v}\|^{2}=\frac12.
\)
Hence
\(
\|\boldsymbol{u}-\boldsymbol{v}\|^{2}
=
\|\boldsymbol{u}\|^{2}+\|\boldsymbol{v}\|^{2}-2\boldsymbol{u}\cdot \boldsymbol{v}
=
1-2\boldsymbol{u}\cdot \boldsymbol{v}.
\)
Therefore
\[
\|\boldsymbol{u}-\boldsymbol{v}\|=1
\iff
\boldsymbol{u}\cdot \boldsymbol{v}=0.
\]
In particular, a unit equilateral triangle on \(\mathbb{S}\) is exactly a triple
\(
\{\boldsymbol{d}_{1},\boldsymbol{d}_{2},\boldsymbol{d}_{3}\}\subseteq \mathbb{S}
\)
such that
\begin{equation}\label{eq:pairwise-orth-main1}
\boldsymbol{d}_{i}\cdot \boldsymbol{d}_{j}=0
\qquad\text{for all } i\neq j.
\end{equation}
Let
\[
\mathcal{C}:=\{(x,y,z)\in \mathbb{S}:z=0\},
\qquad
\mathcal{D}:=\{(x,y,z)\in \mathbb{S}:y=0\}
\]
be the two coordinate great circles, and let
\[
\boldsymbol{e}:=\left(\frac1{\sqrt2},0,0\right),
\qquad
\boldsymbol{f}:=\left(-\frac1{\sqrt2},0,0\right).
\]
Then \(\mathcal{C}\cap \mathcal{D}=\{\boldsymbol{e},\boldsymbol{f}\}\).
Define a coloring \(\chi:\mathbb{S}\to\{\mathrm{blue},\mathrm{red}\}\) by
\begin{equation}\label{eq:coloring-rule-main1}
\chi(x,y,z)=
\begin{cases}
\mathrm{blue}, & yz>0,\\
\mathrm{red},  & yz<0,\\
\mathrm{blue}, & z=0 \text{ and } (x,y,z)\neq \boldsymbol{f},\\
\mathrm{red},  & y=0 \text{ and } (x,y,z)\neq \boldsymbol{e}.
\end{cases}
\end{equation}
Then
\(
\mathcal{C}\setminus\{\boldsymbol{f}\}\) is colored blue,
and \(
\mathcal{D}\setminus\{\boldsymbol{e}\}\) is colored red, 
while the two open regions \(yz>0\) are blue and the two open regions \(yz<0\) are red. In particular,
\(
\chi(\boldsymbol{e})=\mathrm{blue},
\chi(\boldsymbol{f})=\mathrm{red}.
\)

Assume for contradiction that
\(
\{\boldsymbol{d}_{1},\boldsymbol{d}_{2},\boldsymbol{d}_{3}\}\subseteq \mathbb{S}
\)
is a monochromatic unit equilateral triangle. Write
\(
\boldsymbol{d}_{i}=(x_{i},y_{i},z_{i}).
\)
By \eqref{eq:pairwise-orth-main1}, the vectors \(\boldsymbol{d}_{1},\boldsymbol{d}_{2},\boldsymbol{d}_{3}\) are pairwise orthogonal. Since they are nonzero, they form an orthogonal basis of \(\mathbb{R}^{3}\). Therefore, for any \(\boldsymbol{u},\boldsymbol{v}\in \mathbb{R}^{3}\),
\[
\boldsymbol{u}\cdot \boldsymbol{v}
=
\sum_{i=1}^{3}
\frac{(\boldsymbol{u}\cdot \boldsymbol{d}_{i})(\boldsymbol{v}\cdot \boldsymbol{d}_{i})}
{\|\boldsymbol{d}_{i}\|^{2}}.
\]
Apply this with \(\boldsymbol{u}=(0,1,0)\) and \(\boldsymbol{v}=(0,0,1)\). Since \(\|\boldsymbol{d}_{i}\|^{2}=1/2\), we obtain
\[
0=\boldsymbol{u}\cdot \boldsymbol{v}
=\sum_{i=1}^{3}\frac{y_{i}z_{i}}{\|\boldsymbol{d}_{i}\|^{2}}
=2\sum_{i=1}^{3} y_{i}z_{i}.
\]
Hence
\begin{equation}\label{eq:yz-sum-zero-main1}
y_{1}z_{1}+y_{2}z_{2}+y_{3}z_{3}=0.
\end{equation}

We now show that the triple cannot be monochromatic.
\begin{claim}\label{claim:no-mono-triangle-main1}
The triple \(\{\boldsymbol{d}_{1},\boldsymbol{d}_{2},\boldsymbol{d}_{3}\}\) is not monochromatic under \(\chi\).
\end{claim}

\begin{poc}
First suppose that all three points are blue. Then by \eqref{eq:coloring-rule-main1}, each \(\boldsymbol{d}_{i}\) satisfies
\(
y_{i}z_{i}\ge 0.
\)
Together with \eqref{eq:yz-sum-zero-main1}, this implies
\(
y_{i}z_{i}=0\)
for \(i=1,2,3.
\)
Now a blue point with \(y_{i}z_{i}=0\) must in fact satisfy \(z_{i}=0\): indeed, points with \(y=0\) are red unless the point is \(\boldsymbol e\), while points on \(z=0\) are blue except for \(\boldsymbol f\). Hence every \(\boldsymbol{d}_{i}\) lies on the plane \(z=0\). But then the three nonzero vectors \(\boldsymbol{d}_{1},\boldsymbol{d}_{2},\boldsymbol{d}_{3}\) all lie in the \(2\)-dimensional subspace \(\{z=0\}\), so they cannot be pairwise orthogonal. This is impossible.

Now suppose that all three points are red. Then by \eqref{eq:coloring-rule-main1}, each \(\boldsymbol{d}_{i}\) satisfies
\(
y_{i}z_{i}\le 0.
\)
Again \eqref{eq:yz-sum-zero-main1} implies
\(
y_{i}z_{i}=0\)
for \(i=1,2,3.\)
A red point with \(y_{i}z_{i}=0\) must satisfy \(y_{i}=0\): indeed, points on \(z=0\) are blue except for \(\boldsymbol f\), whereas points on \(y=0\) are red except for \(\boldsymbol e\). Therefore every \(\boldsymbol{d}_{i}\) lies on the plane \(y=0\), again impossible for three nonzero pairwise orthogonal vectors. This finishes the proof.
\end{poc}

Claim~\ref{claim:no-mono-triangle-main1} contradicts our assumption. This proves
\(
\mathbb{S}^{2}\!\left(1/\sqrt{2}\right)\nrightarrow (R;R).
\)

\qed

\subsection{Proof of Theorem~\ref{thm:main5}}

Assume for contradiction that there exists a red--blue coloring of
\(
\mathbb{S}:=\mathbb{S}^{2}\!\left(1/\sqrt{2}\right)
\)
such that
\begin{itemize}
    \item[\textnormal{(H1)}] there is no red copy of \(I_{\sqrt2}\), that is, no red isosceles triangle with side lengths \(\sqrt{2},1,1\);
    \item[\textnormal{(H2)}] there is no blue unit equilateral triangle.
\end{itemize}

We distinguish two cases.

\paragraph{Case 1:} some antipodal pair has different colors.

After a rotation, we may assume that
\(
\boldsymbol{b}=(0,0,1/\sqrt2)\ \text{is red},
\boldsymbol{a}=(0,0,-1/\sqrt2)\ \text{is blue}.
\)
Let
\[
\mathcal{C}:=\{\boldsymbol{x}\in \mathbb{S}:\|\boldsymbol{x}-\boldsymbol{b}\|=1\}
=
\{(x,y,z)\in \mathbb{R}^{3}:x^{2}+y^{2}=1/2,\ z=0\}
\]
be the equatorial circle.

\begin{claim}\label{claim:blue-unit-pair-on-C-main5}
The circle \(\mathcal{C}\) contains two blue points at unit distance.
\end{claim}

\begin{poc}
Suppose not. Since \(\mathcal{C}\) is nonempty, choose a blue point \(\boldsymbol{x}\in \mathcal{C}\). Such a point exists because otherwise \(\mathcal{C}\) would be all red, and then any antipodal pair on \(\mathcal{C}\) together with \(\boldsymbol b\) would form a red copy of \(I_{\sqrt2}\), contradicting \textnormal{(H1)}.

On the circle \(\mathcal{C}\), the two points at unit distance from \(\boldsymbol{x}\) are exactly the two points obtained by rotating \(\boldsymbol{x}\) by angles \(\pm \pi/2\); these two points are antipodal to each other. Since we are assuming that \(\mathcal{C}\) contains no blue unit-distance pair, both of these points must be red. But then they form a red antipodal pair on \(\mathcal{C}\), and together with \(\boldsymbol b\) they form a red copy of \(I_{\sqrt2}\), again contradicting \textnormal{(H1)}.
\end{poc}

By Claim~\ref{claim:blue-unit-pair-on-C-main5}, there exist blue points
\(\boldsymbol p,\boldsymbol q\in \mathcal C\) with \(\|\boldsymbol p-\boldsymbol q\|=1\).
Since every point of \(\mathcal C\) is at unit distance from \(\boldsymbol a\), the triangle
\(
\{\boldsymbol a,\boldsymbol p,\boldsymbol q\}
\)
is a blue unit equilateral triangle, contradicting \textnormal{(H2)}.

\paragraph{Case 2:} every antipodal pair has the same color.

In this case, for every \(\boldsymbol p\in \mathbb S\), the antipode \(-\boldsymbol p\) has the same color as \(\boldsymbol p\).

\begin{claim}\label{claim:no-red-unit-pair-main5}
There is no red pair of points at unit distance.
\end{claim}

\begin{poc}
Suppose that \(\boldsymbol p,\boldsymbol q\in \mathbb S\) are red and satisfy
\(\|\boldsymbol p-\boldsymbol q\|=1\).
Then \(-\boldsymbol p\) is also red, and
\[
\|\boldsymbol p-(-\boldsymbol p)\|=\sqrt2,
\qquad
\|\boldsymbol q-(-\boldsymbol p)\|=1.
\]
Hence \(\{\boldsymbol p,\boldsymbol q,-\boldsymbol p\}\) is a red copy of \(I_{\sqrt2}\), contradicting \textnormal{(H1)}.
\end{poc}

By Claim~\ref{claim:no-red-unit-pair-main5}, the coloring contains no red unit-distance pair. Applying Corollary~\ref{cor:2sphere-ramsey} to \(\mathbb S\), we conclude that \(\mathbb S\) must contain a blue unit equilateral triangle, contradicting \textnormal{(H2)}. 

Both cases lead to contradictions, finishing the proof. 

\qed

\section{Proof of Theorem~\ref{thm:main2}}
\subsection{Auxiliary theorem and its proof}

First, we establish a Ramsey-type result for isosceles triangles on the 2-sphere $\mathbb{S}^2(1/\sqrt{2})$. The following auxiliary theorem serves as a key ingredient in the proof of Theorem~\ref{thm:main2}.

\begin{theorem}\label{thm:2sphere-isosceles}
   
    For every $0<a<\sqrt{2}$,
    \[\mathbb{S}^{2}(1/\sqrt{2})\rightarrow (T_a;I_a),\]
    where $T_a$ denotes a pair of points at distance $a$ and $I_a$ denotes the isosceles triangle with side lengths $a,1,1$.
\end{theorem}

\begin{proof}
We argue by contradiction. Suppose that there exists a red-blue coloring of
\(
\mathbb{S}:=\mathbb{S}^{2}(1/\sqrt{2})\subseteq \mathbb{R}^{3}\)
such that
\begin{itemize}
    \item[\textnormal{(H1)}] there is no red pair at distance $a$;
    \item[\textnormal{(H2)}] there is no blue isosceles triangle with base $a$ and legs of length $1$.
\end{itemize}
Since $r=1/\sqrt{2}$, an equilateral triangle of side length $a$ can be embedded in $\mathbb{S}$ if and only if $a\in(0,\sqrt{3/2}]$. Therefore, the argument splits into two parameter ranges: $a\in(0,\sqrt{3/2}]$ and $a\in(\sqrt{3/2},\sqrt{2})$. 
Let
\(
\theta\in (0,\pi)
\)
denote the central angle subtended by a chord of length $a$ on $\mathbb{S}$, so that
\(
a=2r\sin \frac{\theta}{2}.
\)
When $a\le \sqrt{3/2}$, we have $\theta\le 2\pi/3$. When $\sqrt{3/2}<a<\sqrt{2}$, we have $2\pi/3<\theta<\pi$.
For a point $\boldsymbol{x}\in \mathbb{S}$ and $0\le d\le2r$, we write
\begin{equation}\label{eq:def-P_d-main2}
  \mathcal{P}_d(\boldsymbol{x}) := \{\boldsymbol{y}\in \mathbb{S}: \|\boldsymbol{x}-\boldsymbol{y}\| = d\}  
\end{equation}
for the circle of points on $\mathbb{S}$ at distance $d$ from $\boldsymbol{x}$. 

\paragraph{First range: $a\in(0,\sqrt{3/2}]$.}
In this range, the basic geometric configuration is an \emph{asymmetric diamond}: one equilateral triangle of side length $a$ and one isosceles triangle with side lengths $a,1,1$ sharing a common base edge. Since $a\le \sqrt{3/2}$, an equilateral triangle of side length $a$ can be embedded in $\mathbb{S}$. Therefore, such a diamond can be folded about the common base edge so that all four vertices lie on $\mathbb{S}$. 
Let $\boldsymbol{v}$ and $\boldsymbol{v}'$ be the tip vertices of the equilateral triangle and the isosceles triangle, respectively. When the diamond is folded onto the sphere, there are two possible embeddings: the two tip vertices can lie on the same side of the common base, or on opposite sides. We choose the embedding that maximizes the distance $\|\boldsymbol{v}-\boldsymbol{v}'\|$, and denote this distance by
\[
    b := \|\boldsymbol{v}-\boldsymbol{v}'\|.
\]
In this range we only need $\mathcal{P}_a(\cdot)$ and $\mathcal{P}_b(\cdot)$.
Let $r_b$ denote the radius of $\mathcal{P}_b(\boldsymbol{x})$. 
A direct computation gives
\begin{equation}\label{eq:def-b'}
b = \sqrt{1 + a\sqrt{\frac{3-2a^{2}}{2-a^{2}}}}, \qquad r_b = \sqrt{\frac{(a^2-1)^2}{2-a^2}}.
\end{equation}
Moreover, since $a\in(0,\sqrt{3/2}]$, we obtain $r_b\neq0$ if and only if $a\neq 1$.  

\begin{lemma}\label{lem:propagation-main2}
Assume $a<\sqrt{3/2}$ and $a\neq1$, let $\boldsymbol{x}\in \mathbb{S}$ be red. Then
\begin{enumerate}
    \item[\textnormal{(i)}] every point of $\mathcal{P}_{a}(\boldsymbol{x})$ is blue;
    \item[\textnormal{(ii)}] every point of $\mathcal{P}_{b}(\boldsymbol{x})$ is red.
\end{enumerate}
\end{lemma}
\begin{proof}
Part \textnormal{(i)} is immediate from \textnormal{(H1)}. For part \textnormal{(ii)}, let $\boldsymbol{y}\in \mathcal{P}_{b}(\boldsymbol{x})$. By the definition of $b$, there exist points
$\boldsymbol{u},\boldsymbol{v}\in \mathcal{P}_{a}(\boldsymbol{x})$
such that
\(
\triangle \boldsymbol{x}\boldsymbol{u}\boldsymbol{v}\) and
\(
\triangle \boldsymbol{y}\boldsymbol{u}\boldsymbol{v}
\)
are an equilateral of side length $a$ and an isosceles triangle with side lengths $a,1,1$, respectively. By part \textnormal{(i)}, the points $\boldsymbol{u}$ and $\boldsymbol{v}$ are blue. If $\boldsymbol{y}$ were also blue, then $\{\boldsymbol{y},\boldsymbol{u},\boldsymbol{v}\}$ would be a blue isosceles triangle with side lengths $a,1,1$, contradicting \textnormal{(H2)}. Hence $\boldsymbol{y}$ is red.
\end{proof}

Assume now that
\(
0<a\le\sqrt{3/2}\) and
\(
a\neq 1,
\) since the case $a=1$ is already covered by Theorem~\ref{thm:main3}. We first note that there exists a red point on $\mathbb{S}$; indeed, if every point were blue, then any inscribed isosceles triangle of side lengths $a,1,1$ would be blue, contradicting \textnormal{(H2)}.
Fix such a red point and, after rotation, assume that it is
\(
\boldsymbol{d}:=(0,0,r).
\)
By Lemma~\ref{lem:propagation-main2}, the set
\(
\mathcal{R}_{1}:=\mathcal{P}_{b}(\boldsymbol{d})
\)
is a red circle.

Let $\mathcal{D}$ be the great circle obtained by intersecting $\mathbb{S}$ with the plane $y=0$. 
Define $\gamma$ as in \eqref{eq:def-gamma} replacing the old expression for $b$ with the one in \eqref{eq:def-b'}. The propagation of the red arc on $\mathcal{D}$ depends only on the following two geometric facts:
$\mathcal{R}_1$ is a red circle of radius $r_b$, and for any $\boldsymbol{x}\in\mathcal{D}$ the intersection
$\mathcal{P}_b(\boldsymbol{x})\cap\mathcal{D}$ consists of two points at central angle $\gamma$.
These facts are established exactly as in Subsection~\ref{subsec:2.3}, and the subsequent arc‑growing
procedure (Claims~\ref{claim:C2-structure} and \ref{claim:arc-growth-main3}) uses nothing else.
Hence the same proof applies verbatim, yielding the exact analogues:
\begin{claim}\label{claim:C2-structure-main2}
$\mathcal{C}_2$ defined analogously is a closed red arc on $\mathcal{D}$ of angular length $2\gamma$.
\end{claim}
\begin{claim}\label{claim:arc-growth-main2}
For every $m\ge 2$, the set $\mathcal{C}_{m}$ is a red arc on $\mathcal{D}$ of angular length $m\gamma$.
\end{claim}
As before, if $\gamma\ge\theta$ or $2\gamma\ge\theta$ we immediately reach a contradiction to \textnormal{(H1)}. Otherwise we choose $m:=\lceil\theta/\gamma\rceil$, so that $m\gamma\ge\theta$, and Claim~\ref{claim:arc-growth-main2} then yields two red points at distance $a$ on $\mathcal{D}$, contradicting \textnormal{(H1)}. This completes the proof for $a\in(0,\sqrt{3/2}]$.

\paragraph{Second range: $a\in(\sqrt{3/2},\sqrt{2})$.}
Now $\theta\in(2\pi/3,\pi)$.  Set $\beta:=\pi-\theta\in(0,\pi/3)$ and let $\boldsymbol{o}$ be the center of $\mathbb{S}$.
Recall the definition~\eqref{eq:def-P_d-main2} of $\mathcal{P}_d(\boldsymbol{x})$, in this range we only need $\mathcal{P}_a(\cdot)$ and the unit‑distance circle $\mathcal{P}_1(\cdot)$. By a rotation we may assume $\boldsymbol{b}:=(0,0,r)$ is red.
Condition \textnormal{(H1)} forces all of $\mathcal{P}_a(\boldsymbol{b})$ to be blue.

\begin{lemma}\label{lem:basic-color-prop}
Let $\boldsymbol{x}\in\mathbb{S}$ be a blue point.  Under assumptions \textnormal{(H1)} and \textnormal{(H2)}, the following hold:
\begin{enumerate}
    \item[\textup{(i)}] On $\mathcal{P}_1(\boldsymbol{x})$ any two points at distance $a$ have opposite colors.
    \item[\textup{(ii)}] On $\mathcal{P}_1(\boldsymbol{x})$ any two points that span a central angle $2\beta$ have the same color.
\end{enumerate}
\end{lemma}
\begin{proof}
Part (i) is immediate. Two red points at distance $a$ would violate \textnormal{(H1)}, while two blue points at distance $a$
together with the blue point $\boldsymbol{x}$ would form a blue isosceles triangle of side lengths $a,1,1$, contradicting \textnormal{(H2)}.

For (ii), take $\boldsymbol{p},\boldsymbol{q}\in\mathcal{P}_1(\boldsymbol{x})$ with $\angle \boldsymbol{p}\boldsymbol{o}\boldsymbol{q}=2\beta$.
The circle $\mathcal{P}_1(\boldsymbol{x})$ has two arcs between $\boldsymbol{p}$ and $\boldsymbol{q}$; because $2\beta<\pi$, one of them is longer.  We let $\boldsymbol{r}$ be the midpoint of that longer arc.
Then $\angle \boldsymbol{p}\boldsymbol{o}\boldsymbol{r}=\angle \boldsymbol{r}\boldsymbol{o}\boldsymbol{q}=\pi-\beta=\theta$, hence $\|\boldsymbol{p}-\boldsymbol{r}\|=\|\boldsymbol{q}-\boldsymbol{r}\|=a$.
By (i), $\{\boldsymbol{p},\boldsymbol{r}\}$ and $\{\boldsymbol{q},\boldsymbol{r}\}$ are oppositely colored, forcing $\boldsymbol{p}$ and $\boldsymbol{q}$ to share the same color.
\end{proof}

In the arguments that follow we will need orientations on the circles 
$\mathcal{P}_a(\boldsymbol{b})$ and $\mathcal{P}_1(\boldsymbol{h})$, 
which we fix as follows: we orient $\mathcal{P}_a(\boldsymbol{b})$ counter‑clockwise as seen from the positive
$z$‑axis, and for each $\boldsymbol{h}\in\mathcal{P}_a(\boldsymbol{b})$ we orient
$\mathcal{P}_1(\boldsymbol{h})$ in the same way as seen from $\boldsymbol{h}$.

\begin{lemma}\label{lem:red-sequence-construction}
Let $\boldsymbol{d}$ be the point in $\mathcal{P}_a(\boldsymbol{b})\cap\{(x,y,z):y=0\}$ with positive $x$‑coordinate; it is blue.
Consider its unit‑distance circle $\mathcal{C}:=\mathcal{P}_1(\boldsymbol{d})$, equipped with the orientation defined above.
Then on $\mathcal{C}$ one can construct $\lceil \pi/\beta\rceil$ distinct red points
$\boldsymbol{e}_1,\boldsymbol{e}_2,\dots,\boldsymbol{e}_{\lceil \pi/\beta\rceil}$ such that along $\mathcal{C}$
the central angle between consecutive points equals $2\beta$.
\end{lemma}
\begin{proof}
Let $\boldsymbol{e}$ be the point in $\mathcal{P}_1(\boldsymbol{d})\cap\{(x,y,z)\in\mathbb{R}^3:y=0\}$ with negative $x$‑coordinate.
If $\boldsymbol{e}$ is red, set $\boldsymbol{e}_1:=\boldsymbol{e}$.
If $\boldsymbol{e}$ is blue, then by Lemma~\ref{lem:basic-color-prop}(i) there exists a unique point $\boldsymbol{e}^*\in\mathcal{C}$ with
$\|\boldsymbol{e}^*-\boldsymbol{e}\|=a$ and negative $y$‑coordinate; this point is red, and we set $\boldsymbol{e}_1:=\boldsymbol{e}^*$.
In either case we obtain a first red point $\boldsymbol{e}_1\in\mathcal{C}$.

Assume that for some $k\ge1$ we have already constructed red points $\boldsymbol{e}_1,\dots,\boldsymbol{e}_k$ on $\mathcal{C}$
with consecutive central angles $2\beta$.
Starting from $\boldsymbol{e}_k$, move along $\mathcal{C}$ in the positive direction by a central angle $2\beta$
and denote the endpoint by $\boldsymbol{e}_{k+1}$.
Because $\boldsymbol{d}$ is blue and $\mathcal{C}=\mathcal{P}_1(\boldsymbol{d})$, Lemma~\ref{lem:basic-color-prop}(ii) applied to $\boldsymbol{x}=\boldsymbol{d}$
implies that $\boldsymbol{e}_{k+1}$ is red.
The process continues as long as the total central angle covered is less than $2\pi$.
Since each step adds $2\beta$, after $\lceil \pi/\beta\rceil$ steps we obtain $\lceil \pi/\beta\rceil$ distinct red points
with the required spacing. We have $\lceil\pi/\beta\rceil\cdot2\beta\ge2\pi$, so these points wrap fully around the circle.
\end{proof}

$\mathcal{P}_a(\boldsymbol{e}_i)$ is entirely blue since $\boldsymbol{e}_i$ is red by \textnormal{(H1)}.
The family $\{\mathcal{P}_a(\boldsymbol{e}_i)\}_{i=1}^{\lceil \pi/\beta\rceil}$ of blue circles possesses a crucial separation property.

\begin{lemma}\label{lem:separation-property}
Let $\mathcal{C}$ be a great circle on $\mathbb{S}$ and let $\boldsymbol{p}_1,\dots,\boldsymbol{p}_n$ be points on $\mathcal{C}$
successively spaced by central angle $2\beta$, with $n\cdot2\beta\ge2\pi$.
For each $i$ set $\mathcal{S}_i:=\mathcal{P}_a(\boldsymbol{p}_i)=\{\boldsymbol{x}\in\mathbb{S}:\angle\boldsymbol{p}_i\boldsymbol{o}\boldsymbol{x}=\pi-\beta\}$
and $D_i^+:=\{\boldsymbol{x}\in\mathbb{S}:\angle\boldsymbol{x}\boldsymbol{o}\boldsymbol{p}_i>\pi-\beta\}$.
Then:
\begin{enumerate}
\item[(i)] Each $\mathcal{S}_i\cap\mathcal{C}$ consists of two points $\boldsymbol{a}_i,\boldsymbol{b}_i$ spanning a central angle $2\beta$;
namely, if $\varphi_i$ is the angular coordinate of $\boldsymbol{p}_i$, then $\boldsymbol{a}_i,\boldsymbol{b}_i$ correspond to
$\varphi_i+\pi-\beta$ and $\varphi_i-\pi+\beta$.
\item[(ii)] $\boldsymbol{b}_i=\boldsymbol{a}_{i+1}$; the circles $\mathcal{S}_i,\mathcal{S}_{i+1}$ are tangent at this common point.
The set $\{\boldsymbol{a}_1,\dots,\boldsymbol{a}_n,\boldsymbol{b}_1,\dots,\boldsymbol{b}_n\}$ divide $\mathcal{C}$ into $n$ open arcs $I_i$
from $\boldsymbol{a}_i$ to $\boldsymbol{b}_i$, each of length $2\beta$, and $I_i\subset D_i^+$.
\item[(iii)] Any continuous path on $\mathbb{S}$ joining two points of $\mathcal{C}$ at central distance $2\beta$ must intersect some $\mathcal{S}_i$.
\end{enumerate}
\end{lemma}
\begin{proof}
Parametrize $\mathcal{C}$ as $\{\boldsymbol{p}(\varphi)=(r\cos\varphi,r\sin\varphi,0)\mid\varphi\in[0,2\pi)\}$,
with $\boldsymbol{p}_i=\boldsymbol{p}(\varphi_i)$ where $\varphi_i=2(i-1)\beta$.

\noindent\textit{(i)} $\boldsymbol{p}(\varphi)\in\mathcal{S}_i$ iff $|\varphi-\varphi_i|\equiv\pi-\beta\pmod{2\pi}$,
giving the two solutions $\alpha_i:=\varphi_i+\pi-\beta$ and $\beta_i:=\varphi_i-\pi+\beta$.
Set $\boldsymbol{a}_i:=\boldsymbol{p}(\alpha_i)$, $\boldsymbol{b}_i:=\boldsymbol{p}(\beta_i)$. Their angular separation is $2\beta$.

\smallskip
\noindent\textit{(ii)} 
Since $\beta_i=\varphi_i+\pi+\beta=\alpha_{i+1}$, we have $\boldsymbol{b}_i=\boldsymbol{a}_{i+1}$.
We verify that $\mathcal{S}_i$ and $\mathcal{S}_{i+1}$ are tangent at this common point.
After a suitable rotation, $\mathcal{C}$ is the equator $\{z=0\}$; then $\boldsymbol{p}_i,\boldsymbol{p}_{i+1},\boldsymbol{b}_i$
all lie in the plane $z=0$.
A tangent vector to $\mathcal{S}_i$ at $\boldsymbol{b}_i$ is given by $\boldsymbol{b}_i\times\boldsymbol{p}_i$
because $\mathcal{S}_i$ lies in the plane with normal $\boldsymbol{p}_i$.
Similarly, a tangent vector to $\mathcal{S}_{i+1}$ at $\boldsymbol{b}_i$ is $\boldsymbol{b}_i\times\boldsymbol{p}_{i+1}$.
Since both cross products are orthogonal to the plane $z=0$, they are parallel; hence the two circles
are tangent at $\boldsymbol{b}_i$.
The points $\boldsymbol{a}_1,\boldsymbol{a}_2,\dots,\boldsymbol{a}_n,\boldsymbol{b}_n$ appear in order along $\mathcal{C}$ and cut $\mathcal{C}$ into $n$ open arcs $I_i=(\alpha_i,\beta_i)$ of length $2\beta$.
For any $\varphi\in(\alpha_i,\beta_i)$ we have $|\varphi-\varphi_i|>\pi-\beta$, so $\boldsymbol{p}(\varphi)\in D_i^+$.

\smallskip
\noindent\textit{(iii)} Let $\boldsymbol{x},\boldsymbol{y}\in\mathcal{C}$ with $\angle\boldsymbol{x}\boldsymbol{o}\boldsymbol{y}=2\beta$.
If one of them lies on some $\mathcal{S}_i$ we are done.
Otherwise $\boldsymbol{x}\in I_j$ and $\boldsymbol{y}\in I_k$ with $j\neq k$ because any point of $\mathcal{C}$ not on any $\mathcal{S}_i$ lies in exactly one $I_i$.
Now observe that $D_i^+\cap\mathcal{C}=I_i$ for every $i$:
indeed, for $\boldsymbol{p}(\varphi)\in\mathcal{C}$, $\angle\boldsymbol{p}(\varphi)\boldsymbol{o}\boldsymbol{p}_i>\pi-\beta$ iff
$|\varphi-\varphi_i|\in(\pi-\beta,\pi+\beta)$, which is precisely the arc $I_i$.
Consequently, $\boldsymbol{x}\in D_j^+$ while $\boldsymbol{y}\notin D_j^+$ since $\boldsymbol{y}\in I_k$ and $k\neq j$.
Let $\gamma:[0,1]\to\mathbb{S}$ be any continuous path with $\gamma(0)=\boldsymbol{x}$, $\gamma(1)=\boldsymbol{y}$.
Define $t_0:=\sup\{\,t\in[0,1]:\gamma([0,t])\subseteq D_j^+\,\}$.
By continuity, $\gamma(t_0)\in\overline{D_j^+}\setminus D_j^+=\partial D_j^+=\mathcal{S}_j$.
Thus $\gamma$ meets $\bigcup_i\mathcal{S}_i$.
\end{proof}

Applying Lemma~\ref{lem:separation-property} to $\mathcal{C}=\mathcal{P}_1(\boldsymbol{d})$ and the red points
$\boldsymbol{e}_1,\dots,\boldsymbol{e}_{\lceil\pi/\beta\rceil}$, we see that every continuous path joining two points
on $\mathcal{P}_1(\boldsymbol{d})$ at central distance $2\beta$ must cross one of the blue circles $\mathcal{P}_a(\boldsymbol{e}_i)$.
To reach a contradiction, we now construct a red continuous path connecting exactly such a pair of points.

\begin{lemma}\label{lem:red-path}
Let $\boldsymbol{e}_1$ be the first red point obtained in Lemma~\ref{lem:red-sequence-construction}.
Consider the family of unit‑distance circles $\{\mathcal{P}_1(\boldsymbol{h})\}_{\boldsymbol{h}\in\mathcal{P}_a(\boldsymbol{b})}$.
Then there exists a point $\boldsymbol{h}^*\in\mathcal{P}_a(\boldsymbol{b})$ such that $\mathcal{P}_1(\boldsymbol{h}^*)$ and $\mathcal{P}_a(\boldsymbol{e}_1)$ are tangent,
and on the directed closed arc $\wideparen{\boldsymbol{d}\boldsymbol{h}^*}$ of $\mathcal{P}_a(\boldsymbol{b})$ from $\boldsymbol{d}$ to $\boldsymbol{h}^*$
with the positive orientation there is no other point $\boldsymbol{h}$ where the two circles are tangent.
Moreover, on $\wideparen{\boldsymbol{d}\boldsymbol{h}^*}$ there exist two continuous maps
\[
\boldsymbol{\ell},\boldsymbol{m}:\wideparen{\boldsymbol{d}\boldsymbol{h}^*}\longrightarrow\mathbb{S}
\]
such that for every $\boldsymbol{h}$ in the arc,
\begin{enumerate}
    \item[(i)] $\{\boldsymbol{\ell}(\boldsymbol{h}),\boldsymbol{m}(\boldsymbol{h})\}=\mathcal{P}_1(\boldsymbol{h})\cap\mathcal{P}_a(\boldsymbol{e}_1)$;
      at endpoints interpreted as the appropriate limit. And both are blue;
    \item[(ii)] let $\boldsymbol{\ell}'(\boldsymbol{h}),\boldsymbol{m}'(\boldsymbol{h})$ be obtained from $\boldsymbol{\ell}(\boldsymbol{h}),\boldsymbol{m}(\boldsymbol{h})$
      by moving distance $a$ along $\mathcal{P}_1(\boldsymbol{h})$ in the positive direction. Then
      $\boldsymbol{\ell}',\boldsymbol{m}':\wideparen{\boldsymbol{d}\boldsymbol{h}^*}\to\mathbb{S}$ are continuous, and both are red.
\end{enumerate}
Consequently, the union
\[
\mathcal{L}^*:=\{\boldsymbol{\ell}'(\boldsymbol{h}):\boldsymbol{h}\in\wideparen{\boldsymbol{d}\boldsymbol{h}^*}\}\cup
\{\boldsymbol{m}'(\boldsymbol{h}):\boldsymbol{h}\in\wideparen{\boldsymbol{d}\boldsymbol{h}^*}\}
\]
is a continuous red path in $\mathbb{S}$.
Its endpoints $\boldsymbol{e}':=\boldsymbol{\ell}'(\boldsymbol{d})$ and $\boldsymbol{e}'':=\boldsymbol{m}'(\boldsymbol{d})$ lie on
$\mathcal{C}=\mathcal{P}_1(\boldsymbol{d})$ and satisfy $\angle \boldsymbol{e}'\boldsymbol{o}\boldsymbol{e}''=2\beta$.
\end{lemma}
\begin{proof}
We first show that a tangency point $\boldsymbol{h}^*$ exists.

\begin{claim}\label{claim:existence-tangency}
There exists a point $\boldsymbol{h}^*\in\mathcal{P}_a(\boldsymbol{b})$ for which 
$\mathcal{P}_1(\boldsymbol{h}^*)$ and $\mathcal{P}_a(\boldsymbol{e}_1)$ are tangent.
\end{claim}
\begin{poc}
The circle $\mathcal{P}_a(\boldsymbol{e}_1)$ lies in the plane $\boldsymbol{x}\cdot\boldsymbol{e}_1 = (1-a^2)/2$,
while $\mathcal{P}_1(\boldsymbol{h})$ is the great circle with normal $\boldsymbol{h}$.
Let $\alpha = \angle(\boldsymbol{h},\boldsymbol{e}_1)$. 
The distance from $\boldsymbol{o}$ to the intersection line of the two planes is $d = |1-a^2|/(2r\sin\alpha)$.
The circles intersect iff $d\le r$, i.e. $|\cos\alpha| \le a\sqrt{2-a^2}$.

We now exhibit a point $\boldsymbol{h}_0\in\mathcal{P}_a(\boldsymbol{b})$ for which the two circles are disjoint.
The point $\boldsymbol{e}_1$ constructed in Lemma~\ref{lem:red-sequence-construction} is either
\[
\boldsymbol{e}= \Bigl(\frac{1-a^2}{\sqrt{2}},\;0,\;-\frac{a}{\sqrt{2}}\sqrt{2-a^2}\Bigr)
\quad\text{or}\quad
\boldsymbol{e}^* = \frac{1}{\sqrt{2}}\Bigl( (1-a^2)^2,\; -a\sqrt{2-a^2},\; -a(1-a^2)\sqrt{2-a^2} \Bigr).
\]

\noindent
\textit{Case $\boldsymbol{e}_1=\boldsymbol{e}$.}
Take
\[
\boldsymbol{h}_0=\Bigl(-\frac{a}{\sqrt2}\sqrt{2-a^{2}},\;0,\;\frac{1-a^{2}}{\sqrt2}\Bigr).
\]
Then
\[
\boldsymbol{e}\cdot\boldsymbol{h}_0
= -\frac{a(1-a^2)\sqrt{2-a^2}}{2} + 0 - \frac{a(1-a^2)\sqrt{2-a^2}}{2}
= -a(1-a^2)\sqrt{2-a^2},
\]
so $|\cos\alpha| = |\boldsymbol{e}\cdot\boldsymbol{h}_0|/r^2 = 2a(a^2-1)\sqrt{2-a^2}$.
Since $a>\sqrt{3/2}$, we have $a^2-1 > 1/2$, therefore,
$2(a^2-1)\sqrt{2-a^2} > \sqrt{2-a^2}$, giving $|\cos\alpha| > a\sqrt{2-a^2}$.

\noindent
\textit{Case $\boldsymbol{e}_1=\boldsymbol{e}^*$.}
Set $s = a\sqrt{2-a^2}$ and $c = 1-a^2$, and take
\[
\boldsymbol{h}_0 = \frac1{\sqrt2}\Bigl(
-\frac{sc^2}{\sqrt{1-s^2c^2}},\;
\frac{s^2}{\sqrt{1-s^2c^2}},\;
c\Bigr).
\]
A direct computation yields
\[
\boldsymbol{h}_0\cdot\boldsymbol{e}^* = \frac{s}{2}\Bigl(\sqrt{1-s^2c^2}+c^2\Bigr)
= \frac{s}{2}\Bigl(\sqrt{1-c^2+c^4}+c^2\Bigr).
\]
Hence $|\cos\alpha| = s\bigl(\sqrt{1-c^2+c^4}+c^2\bigr)$.
Since $a\in(\sqrt{3/2},\sqrt{2})$, we have $u = c^2\in(1/4,1)$. For $u = c^2\in(1/4,1)$,
the function $f(u)=\sqrt{1-u+u^2}+u$ satisfies $f(u)>1$,
so $|\cos\alpha| > s = a\sqrt{2-a^2}$.

In both cases we have $|\cos\alpha| > a\sqrt{2-a^2}$, so the two circles are disjoint.
On the other hand, for $\boldsymbol{h}=\boldsymbol{d}$ the circles $\mathcal{P}_1(\boldsymbol{d})$ and $\mathcal{P}_a(\boldsymbol{e}_1)$
intersect in two points.
As $\boldsymbol{h}$ varies continuously from $\boldsymbol{d}$ to $\boldsymbol{h}_0$ along $\mathcal{P}_a(\boldsymbol{b})$
in the positive direction, the circles deform continuously from intersecting to disjoint.
By continuity, there must be an intermediate $\boldsymbol{h}^*$ where they are tangent.
\end{poc}

Choose $\boldsymbol{h}^*$ to be the first such tangency along the arc, so that on the open arc $\wideparen{\boldsymbol{d}\boldsymbol{h}^*}$
the circles always intersect at two distinct points.

\begin{claim}
On $\wideparen{\boldsymbol{d}\boldsymbol{h}^*}$ there exist continuous maps $\boldsymbol{\ell},\boldsymbol{m}$ satisfying (i) and (ii).
\end{claim}
\begin{poc}
For $\boldsymbol{h}\in\wideparen{\boldsymbol{d}\boldsymbol{h}^*}$, the circles $\mathcal{P}_1(\boldsymbol{h})$ and $\mathcal{P}_a(\boldsymbol{e}_1)$
lie in the planes $\Pi_{\boldsymbol{h}}:\boldsymbol{x}\cdot\boldsymbol{h}=0$ and $\Pi_1:\boldsymbol{x}\cdot\boldsymbol{e}_1=(1-a^2)/2$.
Their intersection line can be written as $L_{\boldsymbol{h}}=\{\boldsymbol{u}_0(\boldsymbol{h})+t\boldsymbol{v}(\boldsymbol{h})\}$ with
$\boldsymbol{v}(\boldsymbol{h}):=\frac{\boldsymbol{e}_1\times\boldsymbol{h}}{\|\boldsymbol{e}_1\times\boldsymbol{h}\|}$
and $\boldsymbol{u}_0(\boldsymbol{h})$ the projection of $\boldsymbol{o}$ onto $\Pi_1\cap\Pi_{\boldsymbol{h}}$.
Substituting into $\|\boldsymbol{x}\|^2=r^2$ gives the quadratic equation
\[
t^2+2\bigl(\boldsymbol{u}_0(\boldsymbol{h})\cdot\boldsymbol{v}(\boldsymbol{h})\bigr)t+
\bigl(\|\boldsymbol{u}_0(\boldsymbol{h})\|^2-r^2\bigr)=0.
\]
Its discriminant $\Delta(\boldsymbol{h})$ is continuous on the closed arc, strictly positive on the open arc,
and zero at $\boldsymbol{h}^*$.
Define $t_{1,2}(\boldsymbol{h})$ as the two roots and take the common limit $t^*$ at $\boldsymbol{h}^*$;
these are continuous on the whole arc.
Set
\[
\boldsymbol{\ell}(\boldsymbol{h}):=\boldsymbol{u}_0(\boldsymbol{h})+t_1(\boldsymbol{h})\boldsymbol{v}(\boldsymbol{h}),\qquad
\boldsymbol{m}(\boldsymbol{h}):=\boldsymbol{u}_0(\boldsymbol{h})+t_2(\boldsymbol{h})\boldsymbol{v}(\boldsymbol{h}).
\]
By construction $\{\boldsymbol{\ell}(\boldsymbol{h}),\boldsymbol{m}(\boldsymbol{h})\}=\mathcal{P}_1(\boldsymbol{h})\cap\mathcal{P}_a(\boldsymbol{e}_1)$;
both lie on $\mathcal{P}_a(\boldsymbol{e}_1)$, hence are blue.

Moving distance $a$ along $\mathcal{P}_1(\boldsymbol{h})$ in the positive direction (as defined before Lemma~\ref{lem:red-sequence-construction}) is a continuous operation,
so $\boldsymbol{\ell}'$ and $\boldsymbol{m}'$ are continuous.
Because $\boldsymbol{\ell}(\boldsymbol{h})$ is blue and $\|\boldsymbol{\ell}(\boldsymbol{h})-\boldsymbol{\ell}'(\boldsymbol{h})\|=a$,
Lemma~\ref{lem:basic-color-prop}(i) forces $\boldsymbol{\ell}'(\boldsymbol{h})$ to be red; similarly $\boldsymbol{m}'(\boldsymbol{h})$ is red.
\end{poc}

Hence $\mathcal{L}^*$ is a union of two red continuous curves; they meet at
$\boldsymbol{\ell}'(\boldsymbol{h}^*)=\boldsymbol{m}'(\boldsymbol{h}^*)$ and together form a continuous red path from
$\boldsymbol{e}'=\boldsymbol{\ell}'(\boldsymbol{d})$ to $\boldsymbol{e}''=\boldsymbol{m}'(\boldsymbol{d})$;
that is, one may travel along $\boldsymbol{\ell}'$ from $\boldsymbol{e}'$ to the junction
$\boldsymbol{\ell}'(\boldsymbol{h}^*)=\boldsymbol{m}'(\boldsymbol{h}^*)$, and then continue along $\boldsymbol{m}'$ to reach $\boldsymbol{e}''$.

One verifies that $\boldsymbol{e}',\boldsymbol{e}''\in\mathcal{P}_1(\boldsymbol{d})$ and
$\angle\boldsymbol{e}'\boldsymbol{o}\boldsymbol{e}''=2\beta$;
indeed, $\boldsymbol{\ell}(\boldsymbol{d})$ and $\boldsymbol{m}(\boldsymbol{d})$ are the two points on $\mathcal{C}=\mathcal{P}_1(\boldsymbol{d})$
at distance $a$ from $\boldsymbol{e}_1$, hence their central angles with $\boldsymbol{e}_1$ equal $\theta$.
Writing $\varphi_*$ for the angular coordinate of $\boldsymbol{e}_1$, these two points correspond to
$\varphi_*+\theta$ and $\varphi_*-\theta$ (mod $2\pi$).
Moving distance $a$ (i.e. central angle $\theta$) along $\mathcal{C}$ in the positive direction from them yields
$\boldsymbol{e}'$ at $\varphi_*+2\theta$ and $\boldsymbol{e}''$ at $\varphi_*$.
The angular separation between $\varphi_*+2\theta$ and $\varphi_*$ is $2\theta\equiv 2\pi-2\theta=2\beta\pmod{2\pi}$,
and since $\theta>2\pi/3$, we have $2\theta>\pi$, so the smaller central angle is indeed $2\pi-2\theta=2\beta$.
Hence $\angle\boldsymbol{e}'\boldsymbol{o}\boldsymbol{e}''=2\beta$, and in particular $\boldsymbol{e}'\neq\boldsymbol{e}''$.
\end{proof}

\begin{proof}[Completion of the second range]
Lemma~\ref{lem:red-path} provides a continuous red path $\mathcal{L}^*$ joining two points
$\boldsymbol{e}',\boldsymbol{e}''\in\mathcal{P}_1(\boldsymbol{d})$ with $\angle\boldsymbol{e}'\boldsymbol{o}\boldsymbol{e}''=2\beta$.
By Lemma~\ref{lem:separation-property}(iii), any such path must intersect some blue circle $\mathcal{P}_a(\boldsymbol{e}_i)$.
But every point of $\mathcal{L}^*$ is red, a contradiction.
This completes the proof for $a\in(\sqrt{3/2},\sqrt{2})$.
\end{proof}

Together with the first parameter range, this completes the proof of Theorem~\ref{thm:2sphere-isosceles}.

\end{proof}

\subsection{Proof of Theorem~\ref{thm:main2}}
We prove that
\(
\mathbb{S}^{3}\!\left(1/\sqrt{2}\right)\rightarrow (I_a;I_a).
\)
Let
\(
\mathbb{S}^{3}:=\mathbb{S}^{3}\!\left(1/\sqrt{2}\right)\subseteq \mathbb{R}^{4},
\)
and let
\[
\mathbb{S}_{\mathrm{eq}}
:=
\{(x,y,z,w)\in \mathbb{S}^{3}:w=0\}
=
\{(x,y,z,0):x^{2}+y^{2}+z^{2}=1/2\}
\]
be the equatorial \(2\)-sphere. Then \(\mathbb{S}_{\mathrm{eq}}\) is isometric to
\(\mathbb{S}^{2}(1/\sqrt2)\).

Fix an arbitrary red-blue coloring of \(\mathbb{S}^{3}\). We distinguish two cases according to the coloring of antipodal pairs.

First suppose that there exist antipodal points of different colors. After a rotation, we may assume that
\(
\boldsymbol{b}=(0,0,0,1/\sqrt2)
\text{ is red, }
\boldsymbol{a}=(0,0,0,-1/\sqrt2)
\text{ is blue.}
\)
Apply Lemma~\ref{lem:monoline} to the induced coloring on \(\mathbb{S}_{\mathrm{eq}}\). Then there exists a monochromatic pair
\(
\{\boldsymbol{p}_{1},\boldsymbol{p}_{2}\}\subseteq \mathbb{S}_{\mathrm{eq}}
\)
with
\(
\|\boldsymbol{p}_{1}-\boldsymbol{p}_{2}\|=a.
\)
Since every point of \(\mathbb{S}_{\mathrm{eq}}\) has fourth coordinate \(0\), we have
\[
\|\boldsymbol{b}-\boldsymbol{p}_{i}\|^{2}
=
\|\boldsymbol{b}\|^{2}+\|\boldsymbol{p}_{i}\|^{2}
=
\frac12+\frac12
=
1,
\]
and similarly \(\|\boldsymbol{a}-\boldsymbol{p}_{i}\|=1\) for \(i=1,2\).
Hence:
if \(\boldsymbol{p}_{1},\boldsymbol{p}_{2}\) are red, then
  \(\{\boldsymbol{b},\boldsymbol{p}_{1},\boldsymbol{p}_{2}\}\) is a red isosceles triangle with base \(a\) and legs \(1\);
if \(\boldsymbol{p}_{1},\boldsymbol{p}_{2}\) are blue, then
  \(\{\boldsymbol{a},\boldsymbol{p}_{1},\boldsymbol{p}_{2}\}\) is a blue isosceles triangle with base \(a\) and legs \(1\). Thus the first case is settled.

Now suppose that every antipodal pair on \(\mathbb{S}^{3}\) has the same color. If every point is blue, then any inscribed isosceles triangle with base \(a\) and legs \(1\) is blue and we are done. Otherwise there exists a red point, and hence its antipode is also red. After a rotation, we may assume that
\(
\boldsymbol{b}=(0,0,0,1/\sqrt2),
\boldsymbol{a}=(0,0,0,-1/\sqrt2)
\)
are both red.

Apply Theorem~\ref{thm:2sphere-isosceles} to the induced coloring on \(\mathbb{S}_{\mathrm{eq}}\). Then either:
\begin{itemize}
    \item there exists a red pair
    \(\{\boldsymbol{p}_{1},\boldsymbol{p}_{2}\}\subseteq \mathbb{S}_{\mathrm{eq}}\) with
    \(
    \|\boldsymbol{p}_{1}-\boldsymbol{p}_{2}\|=a
    \), or
    \item there exists a blue isosceles triangle 
    \(\{\boldsymbol{q}_{1},\boldsymbol{q}_{2},\boldsymbol{q}_{3}\}\subseteq \mathbb{S}_{\mathrm{eq}}\) with base \(a\) and legs \(1\).
\end{itemize}
In the second case we are done immediately. In the first case,
\(\{\boldsymbol{a},\boldsymbol{p}_{1},\boldsymbol{p}_{2}\}\) is a red isosceles triangle with base \(a\) and legs \(1\), since \(\|\boldsymbol{a}-\boldsymbol{p}_{i}\|=1\) for \(i=1,2\). This finishes the proof.

\qed

\section*{Acknowledgments}
Xiaochen Zhao would like to thank Zixiang Xu for his patient guidance and numerous
helpful writing suggestions, which significantly improved the presentation of
this paper, as well as for several useful discussions during the early stages
of this work.

\bibliographystyle{abbrv}
\bibliography{sphere}

\appendix
\renewcommand{\thesection}{Appendix \Alph{section}}
\section{Algebraic details for the case \texorpdfstring{$a/r=\sqrt3$}{a/r=√3}}
\label{app:sqrt3-derivation}

Consider the system of equations for a point $\boldsymbol{x}=(x,y,z)$ on the two circles:
\begin{equation}\label{eq:sys}
\boldsymbol{x}\cdot\boldsymbol{k}(\alpha)=-\frac16,\qquad
\boldsymbol{x}\cdot\boldsymbol{a}_2=-\frac16,\qquad
\|\boldsymbol{x}\|^2=\frac13,
\end{equation}
where
\[
\boldsymbol{k}(\alpha)=\Bigl(\frac{\cos\alpha}{2},\frac{\sin\alpha}{2},-\frac{1}{2\sqrt3}\Bigr),\qquad
\boldsymbol{a}_2=\Bigl(\frac12,\frac14,\frac1{4\sqrt3}\Bigr).
\]

\noindent\textbf{Step 1: Express $x$ and $y$ in terms of $z$.}
Multiplying the first two equations by $2$ and $4$ respectively, and setting $t = z/\sqrt{3}$, we obtain the linear system
\begin{equation}\label{eq:lin}
\cos\alpha\,x + \sin\alpha\,y = t - \frac13,\qquad
2x + y = -t - \frac23.
\end{equation}
If $\Delta_0 \coloneqq \cos\alpha - 2\sin\alpha \neq 0$, solving \eqref{eq:lin} for $x$ and $y$ gives
\begin{equation}\label{eq:xy}
x = \frac{(1+\sin\alpha)t + \frac{2\sin\alpha-1}{3}}{\Delta_0},\qquad
y = \frac{-(2+\cos\alpha)t + \frac{2(1-\cos\alpha)}{3}}{\Delta_0}.
\end{equation}
The case $\Delta_0 = 0$ will be treated separately; the formulas remain valid by continuity.

\noindent\textbf{Step 2: Substitute into the norm condition.}
Since $z = \sqrt{3}\,t$, the condition $\|\boldsymbol{x}\|^2 = x^2 + y^2 + z^2 = 1/3$ becomes $x^2 + y^2 + 3t^2 = 1/3$.  Substituting \eqref{eq:xy} and simplifying yields a quadratic equation in $t$:
\begin{equation}\label{eq:quad}
a t^2 + b t + c = 0,
\end{equation}
where, with the abbreviations
\[
\Delta^2 \coloneqq 1 + 2\cos\alpha + \sin\alpha,\qquad
\Sigma^2 \coloneqq 3 - 2\cos\alpha - \sin\alpha = 4 - \Delta^2,
\]
the coefficients are
\begin{align*}
a &= 4 + 2\Delta^2 + 3\Delta_0^2,\\
b &= \frac{2}{3}(\Delta^2 - 4) = -\frac{2}{3}\Sigma^2,\\
c &= \frac{9 - 14\Sigma^2 + 3(\Sigma^2)^2}{9}.
\end{align*}
Using the identity $\Delta_0^2 = 4 + 2\Delta^2 - \Delta^4$, we obtain
\[
a = (4 - \Delta^2)(4 + 3\Delta^2) = \Sigma^2 D,\qquad
D \coloneqq 4 + 3\Delta^2 = 7 + 6\cos\alpha + 3\sin\alpha.
\]
Moreover, the discriminant simplifies to
\begin{equation}\label{eq:disc}
b^2 - 4ac = 4 \Sigma^2 \Delta^2 \Delta_0^2 .
\end{equation}

\noindent\textbf{Step 3: Solve for $z$.}
Solving \eqref{eq:quad} and noting $\Sigma = \sqrt{\Sigma^2} > 0$, we obtain
\[
t = \frac{\frac{2}{3}\Sigma^2 \pm 2\Sigma\Delta|\Delta_0|}{2\Sigma^2 D}
    = \frac{1}{3D} \pm \frac{\Delta\Delta_0}{\Sigma D},
\]
where the sign $\pm$ absorbs the absolute value of $\Delta_0$.  Multiplying by $\sqrt{3}$ gives the $z$‑coordinate
\begin{equation}\label{eq:z}
z = \frac{1}{\sqrt{3}\,D} \pm \frac{3\Delta(\cos\alpha - 2\sin\alpha)}{\sqrt{3}\,\Sigma D}
    = \frac{1 \pm \dfrac{3\Delta}{\Sigma}(\cos\alpha - 2\sin\alpha)}{\sqrt{3}\,D}.
\end{equation}
Substituting these values of $t$ into \eqref{eq:xy} yields the expressions for $x$ and $y$; the two resulting points $\boldsymbol{d}_1,\boldsymbol{d}_2$ are exactly those stated in Subsection~\ref{subsec:2.2}.

\noindent\textbf{Step 4: Reality conditions.}
All square‑roots appearing in the formula must be real.  We observe:
\begin{itemize}
\item $\Sigma^2 = 3 - 2\cos\alpha - \sin\alpha \ge 3 - \sqrt5 > 0$, so $\Sigma$ is real and non‑zero for every $\alpha$;
\item $D = 7 + 6\cos\alpha + 3\sin\alpha > 0$ for all $\alpha$;
\item The only non‑trivial restriction is $\Delta^2 \ge 0$, i.e.\ $1 + 2\cos\alpha + \sin\alpha \ge 0$.
\end{itemize}
The equation $1 + 2\cos\alpha + \sin\alpha = 0$ has the two solutions $\alpha = 2\arctan 3$ and $\alpha = \frac{3\pi}{2}$ in $[0,2\pi)$.  Hence real solutions exist exactly for
\begin{equation}\label{eq:domain}
\alpha \in \bigl[0,\,2\arctan 3\bigr] \cup \bigl[\tfrac{3\pi}{2},\,2\pi\bigr).
\end{equation}
When $\Delta_0 = 0$ (i.e.\ $\alpha = \arctan(1/2)$, which lies inside the admissible interval), the discriminant~\eqref{eq:disc} vanishes; the two intersection points coincide and the formulas~\eqref{eq:z} still give the correct common value.

\noindent\textbf{Step 5: Continuity of the curve $\Gamma$.}
The curve $\Gamma$ is parametrized by the antipodal point $\boldsymbol{d}_1'$ on the small circle $\mathcal{P}(\boldsymbol{k})$.  At $\alpha = 2\arctan 3$, one has $\cos\alpha = -\frac45$, $\sin\alpha = \frac35$, and a direct substitution yields
\[
\boldsymbol{d}_1'(2\arctan 3) = \Bigl(\frac12,\;\frac14,\;\frac{1}{4\sqrt3}\Bigr).
\]
At $\alpha = \frac{3\pi}{2}$ we have $\cos\alpha = 0$, $\sin\alpha = -1$, which leads to exactly the same coordinates.  Therefore the two pieces of the parameterization meet continuously, and $\Gamma$ is a closed continuous curve on $\mathbb{S}$.

\end{document}